\documentclass[a4paper,11pt]{article}

\usepackage{amsmath,mathtools,amsthm,amssymb,amsfonts,color,comment}
\usepackage[numbers,sort&compress]{natbib}
\usepackage[colorlinks,linkcolor=blue,anchorcolor=blue,citecolor=green]{hyperref}
\usepackage{enumerate}
\numberwithin{equation}{subsection}

\newcommand{\mi}{(\mathbf{I-P})}
\allowdisplaybreaks[4]

\newtheorem{theorem}{Theorem}[section]

\newtheorem{lemma}{Lemma}[section]
\newtheorem{remark}{Remark}[section]

\newtheorem{corollary}{Corollary}[section]
\newtheorem{proposition}{Proposition}[section]
\usepackage{mathrsfs}
\makeatletter

\newcommand{\Rmnum}[1]{\expandafter\@slowromancap\romannumeral #1@}
\usepackage{authblk}
\usepackage[title]{appendix}
\allowdisplaybreaks[4]

\def\R{\mathbb{R}}

\def\v{\varepsilon}
\def\div{\mathrm{div}}

\def\l{\langle}
\def\r{\rangle}

\def\up{\textup}
\def\n{\nabla}
\def\p{\partial}

\def\a{\alpha}
\def\b{\beta}

\def\le{\lesssim}

\def\c{\cdot}

\def\E{\mathbb{E}}
\def\M{\sqrt{M_u}}
\def\d{\mathrm{d}}

\renewcommand{\div}{\mathrm{div}}
\newcommand{\pt}{\partial_{t}}
\newcommand{\px}{\partial_{x}}

\newcommand{\intt}[2]{\int_{0}^{t}{\, #1}\, \mathrm{d}{#2}}

\newcommand{\Ltwox}[1]{\Vert {#1} \Vert_{L^2_x}}
\newcommand{\Ltwov}[1]{\Vert {#1} \Vert_{L^2_v}}
\newcommand{\Ltwoxv}[1]{\Vert {#1} \Vert_{L^2_xL^2_v}}

\newcommand{\Htwox}[1]{\Vert{#1} \Vert_{H^2_x}}
\newcommand{\Hthreex}[1]{\Vert{#1} \Vert_{H^3_x}}

\newcommand{\dt}[1]{\frac{\mathrm{d}}{\mathrm{d}#1}}

\newcommand{\intx}[1]{\int {#1} \,\mathrm{d}{x}}
\newcommand{\intv}[1]{\int {#1} \,\mathrm{d}{v}}
\newcommand{\intxv}[1]{\int {#1} \,\mathrm{d}{x} \,\mathrm{d}{v}}

\newcommand{\Lx}[1]{\Vert {#1} \Vert_{L^2_{x}}}

\newcommand{\Hxx}[1]{\Vert{#1} \Vert_{H^2_{x}}}
\newcommand{\Hxxx}[1]{\Vert{#1} \Vert_{H^3_{x}}}
\newcommand{\Hxxxx}[1]{\Vert{#1} \Vert_{H^4_{x}}}
\newcommand{\Lxv}[1]{\Vert {#1} \Vert_{L^2_{x}L^2_{v}}}

\newcommand{\Lv}[1]{\Vert {#1} \Vert_{\nu}}
\newcommand{\Linf}[1]{\Vert {#1} \Vert_{L^\infty_{x}}}
\newcommand{\Hxv}[1]{\Vert{#1} \Vert_{H^1_{x}L^2_{v}}}
\newcommand{\Hxxv}[1]{\Vert{#1} \Vert_{H^2_{x}L^2_{v}}}
\newcommand{\Hxxxv}[1]{\Vert{#1} \Vert_{H^3_{x}L^2_{v}}}
\newcommand{\Hxxxxv}[1]{\Vert{#1} \Vert_{H^4_{x}L^2_{v}}}
\newcommand{\Hxxvv}[1]{\Vert{#1} \Vert_{H^2_{x}H^2_{v}}}

\newcommand{\HxxLv}[1]{\Vert{#1} \Vert_{H^2_xL^2_v}}
\newcommand{\HxxxLv}[1]{\Vert{#1} \Vert_{H^3_xL^2_v}}
\newcommand{\HxxxxLv}[1]{\Vert{#1} \Vert_{H^4_xL^2_v}}

\newcommand{\Hfourx}[1]{\Vert{#1} \Vert_{H^4_x}}

\newcommand{\Linftyx}[1]{\Vert{#1} \Vert_{L^\infty_x}}

\newcommand{\inx}[1]{\langle {#1} \rangle_{x}}

\newcommand{\inxv}[1]{\langle {#1} \rangle_{x,v}}

\usepackage{color}

\makeatother

\begin{document}
	
	\title{A Quantitative Hydrodynamic Limit for the Vlasov--Fokker--Planck/Navier--Stokes System in the Fine Particle Regime}

	\author[a]{Zhendong Fang}
	\author[b]{Jian Li}
	\author[c]{Kunlun Qi}
	\author[b]{Huanyao Wen\thanks{Corresponding author. Email: mahywen@scut.edu.cn.}}
	
	\affil[a]{School of Mathematics and Information Science, Guangzhou University, Guangzhou, 510006, P.~R.~China}
	\affil[b]{School of Mathematics, South China University of Technology, Guangzhou, 510641, P.~R.~China}
	\affil[c]{Department of Computational Mathematics, Science and Engineering and Department of Mathematics, Michigan State University, East Lansing, MI 48824, USA.}
	
	\date{}
	\maketitle
	\renewcommand{\thefootnote}{}

   \begin{abstract}
    We justify the hydrodynamic limit of the Vlasov--Fokker--Planck equation coupled with the incompressible Navier--Stokes equation in the fine regime with a quantitative convergence rate. In this regime, the limiting dynamics is described by the inhomogeneous incompressible Navier--Stokes system for the fluid--particle mixture.
    In contrast with the existing relative entropy approaches,
    %we provide the first rigorous justification of this limit in a pointwise strong sense
    the present work extends to a pointwise convergence of the hydrodynamic limits with an explicit rate on any given time interval $[0,T]$.
    The proof is based on a Macro--Micro decomposition around the moving local Maxwellian, together with an $\varepsilon$-weighted energy--dissipation structure adapted to the fine particle relaxation scales.
    Furthermore, we remove the classical small background restriction inherent in high-order energy methods, allowing the limiting profile to be possibly large. This is accomplished by exploiting a hidden Kelvin--Voigt type damped wave equation structure within the macroscopic velocity error system, yielding the critical extra dissipation required to close the uniform estimates.
\end{abstract}

	\vspace{4mm}
	
{\noindent \textbf{Keywords:} Hydrodynamic limit, Vlasov--Fokker--Planck/Navier--Stokes System, Fine particle,  Inhomogeneous incompressible Navier--Stokes equation.}
	
\vspace{4mm}
{\noindent\textbf{AMS Subject Classification (2020):} 35Q83, 35Q84, 35Q30, 35B40}

\tableofcontents

	%%%%%%%%%%%%%%%%%%%%%%%%%%%%%%%%%%%%%%%%%%%%%%%%%
	\section{Introduction}
	\label{section_1}
	
	\subsection{The model}
	\label{subsec:model}
	
	In this paper, we investigate the hydrodynamic limit of a coupled kinetic–fluid model that describes the collective dynamics of fine particles dispersed within an incompressible viscous fluid. This type of system is fundamental in the study of aerosols, sprays, and sedimentation processes \cite{SBRBKT03,WFA85}.
	
	The state of the particles is described by a non-negative density distribution function $f(t,x,v)$, representing the density of particles at time $t > 0$ and position $x\in \mathbb{R}^3 $ with velocity $v\in \mathbb{R}^3 $, whose evolution follows the Vlasov--Fokker--Planck (VFP) equation
	\begin{equation*}
	\p_t f + v \cdot \n_x f + \mathrm{div}_v (F_d f - r\n_v f) = 0.
	\end{equation*}
	where $r>0$ is the diffusion coefficient and $F_d$ denotes the acceleration exerted by the fluid on the particles through a drag force.
	
	The surrounding fluid is characterized by its velocity field $u(t,x) \in \mathbb{R}^3$ and pressure $P(t,x)$. Assuming the particle concentration is sufficiently dilute such that the fluid density remains constant, the fluid evolves according to the incompressible Navier--Stokes (INS) system:
	\begin{equation*}
	\left\{
	\begin{aligned}
	\p_t u+u \cdot \n_x u+\n_x P-\Delta_x u &= \mathcal{F},\\[5pt]
	\mathrm{div}_xu &= 0,
	\end{aligned}
	\right.
	\end{equation*}
	where $\mathcal{F}$ represents the momentum feedback from the particles onto the fluid.
	
	In fact, the coupling between the kinetic and fluid phases is mediated by the friction (or drag) force $F_d$. We adopt the standard linear closure where the Stokes drag is proportional to the relative velocity between the fluid and the particles:
	\[
	F_d = F_0 \big( u(t,x) - v \big),
	\]
	where $F_0 > 0$ is the drag coefficient. By Newton's third law, the feedback force $\mathcal{F}$ acting on the fluid is the negative of the total drag exerted on the particles, integrated over the velocity space:
	\[
	\mathcal{F} = - \int_{\mathbb{R}^3} F_df \,\d v = F_0 \int_{\mathbb{R}^3} f(v-u) \,\d v.
	\]

	After non-dimensionalization and adopting the fine-particle scaling of \cite{Goudon2004fine}, the scaled Vlasov--Fokker--Planck/incompressible Navier--Stokes (VFP-INS) system reads:
	\begin{equation}\label{f}
	\left\{
	\begin{aligned}
	&\p_t f^\v+v\cdot \n_xf^\v = \ \frac{1}{\v}\up{div}_v\big[\n_vf^\v+(v-u^\v)f^\v\big],\\[5pt]
	&\p_t u^\v+u^\v \cdot \n_x u^\v+\n_x P^\v = \ \Delta_x u^\v+\frac{1}{\v}\intv{(v-u^\v)f^\v},\\[5pt]
	&\mathrm{div}_x u^\v = \ 0,\\[5pt]
	&(f^\v(t,x,v),u^\v(t,x))|_{t=0}=(f^{\v,in}(x,v),u^{\v,in}(x)).\\[5pt]
	\end{aligned}
	\right.
	\end{equation}
  The initial data $\big(f^{\v,in}(x,v),u^{\v,in}(x) \big)$ satisfy the far-field behavior
	\begin{equation*}
	\big( f^{\v,in}(x,v),u^{\v,in}(x) \big)\rightarrow(M,0), \quad \text{as} \quad |x|\rightarrow +\infty,
	\end{equation*}
	where $M$ is the normalized global Maxwellian equilibrium
	\begin{equation*}\label{Mv}
	M=M(v):=\frac{1}{(2\pi)^{3/2}} \up{e}^{-\frac{|v|^2}{2}}.
	\end{equation*}
	
	The hydrodynamic limit of \eqref{f} corresponds to the regime $\v \rightarrow 0$, where the particle distribution $f^\v$ is expected to relax toward a local Maxwellian centered at the fluid velocity $u^\v$. Formally, this leads to a macroscopic description of the particle density fluctuation $\rho_0(t,x)$ and the fluid velocity $u_0(t,x)$, and the resulting limit system is the inhomogeneous incompressible Navier--Stokes (IINS) system \cite{Lions1996, Goudon2004fine} for $\tilde{\rho}_0=\rho_0+2$:
\begin{equation}\label{limit}
\left\{
\begin{aligned}
&\p_t\tilde{\rho}_0+u_0\cdot \n_x \tilde{\rho}_0 = \ 0,\\[5pt]
&\tilde{\rho}_0 \p_t u_0+\tilde{\rho}_0 u_0\cdot \n_xu_0+\n_x P_0-\Delta_xu_0 = \ 0,\\[5pt]
&\mathrm{div}_xu_0 = \ 0,\\[5pt]
&\big(\tilde{\rho}_0(t,x),u_0(t,x)\big)|_{t=0}=\big(\tilde{\rho}^{in}_0(x),u^{in}_0(x)\big).\\[5pt]
\end{aligned}
\right.
\end{equation}
This system describes the macroscopic dynamics of the fluid--particle mixture, where the total inertia of the mixture depends on the local particle concentration.

The main purpose of this paper is to rigorously justify the hydrodynamic limit from the strong solution of the scaled VFP-INS system \eqref{f} to the strong solution of the IINS system \eqref{limit} with an explicit convergence rate.
	
	%--------------------------------------
	\subsection{Previous results and main challenges}
	\label{subsec:previous}

	To clearly present the challenges of this paper, we first recall some results on the well-posedness and hydrodynamic limits of kinetic--fluid systems involving the Vlasov--Fokker--Planck equation coupled with fluid equations through drag interactions. We then explain the main analytical difficulties specific to the fine particle regime and describe our new ideas to overcome them.
	\smallskip

	\textit{Previous results on well-posedness.}
	The well-posedness theory for Vlasov--Fokker--Planck equation coupled with Navier--Stokes type systems has been extensively investigated. For the Vlasov--Fokker--Planck equation coupled with incompressible Navier--Stokes system, Goudon--He--Moussa--Zhang \cite{Goudon2010} established the global-in-time existence of classical solutions near equilibrium, together with further regularity properties and long-time behavior. Chae--Kang--Lee \cite{ChaeMyeongju2011} proved the existence of weak solutions in both two and three space dimensions and obtained uniqueness of smooth solutions in the two-dimensional case. For the corresponding diffusion-free Vlasov/Navier--Stokes system, Yu in \cite{Yu13} proved global weak existence in two and three dimensions and uniqueness in two dimensions, and the similar result was extended by Wang-Yu to the inhomogeneous Navier-Stokes-Vlasov equations in \cite{WY2015}.

	For compressible kinetic--fluid couplings, Mellet--Vasseur \cite{Mellet2007} established the existence of global weak solutions to the Vlasov--Fokker--Planck equation coupled with the compressible barotropic Navier--Stokes system in bounded domains under physically relevant boundary conditions.
    In a different asymptotic direction, we also mention the work of Flandoli--Leocata--Ricci \cite{Flandoli2021}, which provides a microscopic derivation of the Vlasov--Fokker--Planck/Navier--Stokes system from an interacting particle--fluid model, proving convergence of empirical measures to the kinetic density and convergence of the fluid velocity in a uniform topology.
	These works provide an analytical and modeling foundation for the study of kinetic--fluid coupled systems.
	
	\smallskip
	
	\textit{Previous results on hydrodynamic limits.}
	The hydrodynamic limit of the Vlasov--Fokker--Planck/Navier--Stokes system was systematically studied in the pioneering two-part work of Goudon--Jabin--Vasseur \cite{Goudon2004light, Goudon2004fine}, where two different asymptotic regimes were identified according to the relative scaling of the particle and fluid densities. In the light particle regime \cite{Goudon2004light}, the limiting system is the incompressible Navier--Stokes equation coupled with a Kramers--Smoluchowski type equation for the particle density. In the fine particle regime \cite{Goudon2004fine}, where the particle density is of the same order as the fluid density, the limiting dynamics is governed by the inhomogeneous incompressible Navier--Stokes system for the fluid--particle mixture.
     This distinction is also reflected in the relevant kinetic equilibrium: the light particle regime is associated with a fixed centered Maxwellian velocity profile, whereas the fine particle regime involves a $t,x$-dependent local Maxwellian centered at the fluid velocity. The latter dependence becomes particularly significant in a differentiated high-order analysis.
	As a consequence, the methods and convergence topologies in these two regimes are substantially different. In the light particle regime, the proof relies on free-energy estimates and compactness arguments. In the fine particle regime, Goudon--Jabin--Vasseur \cite{Goudon2004fine} employed the relative entropy method to compare weak solutions of the scaled Vlasov--Fokker--Planck/Navier--Stokes system with smooth solutions of the limiting inhomogeneous incompressible Navier–Stokes system. This yields convergence of the particle density in an \(L^\infty(0,T;L^1)\)-type topology and convergence of the fluid velocity in an \(L^\infty(0,T;L^2)\)-type topology on any fixed time interval.

	Several extensions and related results have subsequently appeared. Su--Yao \cite{YunfeiSu2020} studied hydrodynamic limits for inhomogeneous incompressible Navier--Stokes/Vlasov--Fokker--Planck systems, typically in bounded or periodic domains, using relative entropy type approaches in different functional settings. Recently, Choi--Jung \cite{ChoiJung2026} explicitly developed a quantitative relative entropy theory for the light particle VFP-INS limit. In addition, Han-Kwan-Michel proposed a framework to justify hydrodynamic limits of the Vlasov--Navier--Stokes system in \cite{Han-Kwan-Michel2024}. For the Vlasov--Fokker--Planck equation coupled with the compressible Navier--Stokes system, Mellet--Vasseur \cite{Mellet2008} justified hydrodynamic limits in bounded domains with Dirichlet or reflection boundary conditions, building on their earlier well-posedness work \cite{Mellet2007}.

    Relative entropy methods are quite useful and naturally yield weak-strong stability and strong convergence in integral norms. Quantitative relative entropy inequalities may also yield explicit modulated energy bounds once the initial relative entropy is quantified.
    To obtain stronger convergence, Fang--Qi--Wen \cite{Fang-Qi-Wen2024} recently extended the result in the light particle regime of the VFP-INS system by a Hilbert expansion approach and derived pointwise convergence with an explicit rate.
    For hydrodynamic limits of other kinetic models, including the Boltzmann equation and related kinetic--fluid systems, we refer the reader to \cite{Bardos1991-2} and the references therein.

	\smallskip
	%--------------------------
	
	\smallskip
	\textit{Mathematical challenges.}

    The purpose of the present paper is to establish a quantitative strong solution theory for the hydrodynamic limit of VFP-INS system in fine particle regime. More precisely, we compare a strong solution of the scaled VFP-INS system \eqref{f}  with a prescribed smooth solution of the IINS system \eqref{limit} and prove pointwise convergence in \(L^\infty_{x,v}\times L^\infty_x\) with an explicit rate. The result is therefore complementary to the relative entropy theory of \cite{Goudon2004fine}, which applies at the finite energy weak solution level and gives strong convergence in natural integral norms, whereas the present work assumes the initial condition with high-regularity and a smooth limiting profile (``strong solution"-to-``strong solution") and obtain pointwise estimate with quantitative convergence rate. Beyond strengthening the convergence topology, the reason why we investigate this pointwise theory is also motivated in part by the future analysis of boundary effects in coupled kinetic--fluid models.
    Compared with the Hilbert-expansion framework developed for the light particle regime in \cite{Fang-Qi-Wen2024}, the fine particle scaling considered here leads to a different limiting structure and requires a substantially different energy method. The main challenges and novelties are presented as follows.

	\begin{itemize}

\item \textbf{A $t,x$-dependent Macro-Micro decomposition around a local Maxwellian equilibrium.}
In previous studies on global well-posedness near equilibrium \cite{Goudon2010} and on the hydrodynamic limit of light particle \cite{Fang-Qi-Wen2024}, the Macro--Micro decomposition is performed around the global Maxwellian distribution. This choice is natural in the light particle scaling, where the leading equilibrium is spatially homogeneous. In the fine particle regime, however, the dominant Fokker--Planck relaxation drives the particle distribution toward
\[
M_{u^\v}(t,x,v)
=
\frac{1}{(2\pi)^{3/2}}
\exp\!\left(-\frac{|v-u^\v(t,x)|^2}{2}\right),
\]
whose center is the unknown fluid velocity. Accordingly, we write
\[
f^\varepsilon = M_{u^\v} + g^{\v} \sqrt{M_{u^\v}},
\qquad \text{with} \quad
g^{\v}=\mathbf P g^{\v} + (\mathbf I-\mathbf P)g^{\v}.
\]
A local Maxwellian also appears in the relative entropy comparison of \cite{Goudon2004fine}; the novelty here is not the local Maxwellian itself, but its use in a differentiated high-order macro--micro theory. Indeed, \(M_{u^{\v}}\), the projection \(\mathbf P\), and the linearized Fokker--Planck operator all depend nonlinearly on \(u^{\v}(t,x)\). Spatial differentiation of the singular relaxation term therefore produces \(\varepsilon^{-1}\)-weighted commutators involving derivatives of \(u^{\v}\), together with additional couplings between the macroscopic moments and the microscopic component. These terms are absent when the velocity profile is a fixed global Maxwellian, as in the light particle analysis of \cite{Fang-Qi-Wen2024}.
To control the extra terms, we introduce the following two new types of estimates:
\begin{itemize}
\item \textbf{Macroscopic error estimate.} We instead introduce the macroscopic errors functions
\[
\phi:=a-\rho_0,\quad \psi:=u^{\v}-u_0,
\]
with density $a$ defined in \eqref{eq:def-a-b}, and then derive the corresponding coupled system \eqref{eq:error}, together with the microscopic component. This formulation separates the possibly large limiting profile from the genuinely small error variables and rapidly relaxing terms.

\item \textbf{$\v$-weighted energy estimate.} We construct an \(\varepsilon\)-weighted energy estimate based on the Macro-Micro decomposition. Its weights reflect the different relaxation scales of the density, the momentum, the fluid velocity, and the microscopic component. The resulting energy--dissipation structure is adapted to the fine particle relaxation mechanism and absorbs the singular commutators generated by the differentiated moving Maxwellian.
\end{itemize}

\item \textbf{Possibly large regular limiting profile.}
Our results hold as long as the limiting system \eqref{limit} is well-posed with sufficiently high regularity on an arbitrary time interval $[0,T]$, requiring no smallness conditions on its solution.
Specifically, the limiting solution is retained as a prescribed variable coefficient, the weights of which, at the highest derivative level,
\begin{equation}\label{eq:weighted-h-u}
    \varepsilon^2
\|\nabla_x h\|_{H_x^3L_v^2}^2
+
\varepsilon^3
\|\nabla_x^2 u^{\v}\|_{H_x^3}^2,
\qquad h = g^{\v} + \sqrt{M_{u^{\v}}} = \frac{f^\varepsilon}{\sqrt{M_{u^{\v}}}}
\end{equation}
match the singular relaxation scale and inherently allow the corresponding unweighted norms $\|\nabla_x h\|_{H^3_xL^2_v}^2$ and $\|\nabla_x^2 u^{\v}\|_{H^3_x}^2$ to be large, fully accommodating the large background.

With non-smallness condition, the leading macroscopic variables are not necessarily small: one expects \(a \approx \rho_0\) and \(u^{\v} \approx u_0\), and the prescribed limiting solution \((\rho_0,u_0)\) may have large high-order Sobolev norms on a certain time interval.
This leads to the following two critical mathematical challenges in both the error estimate and $\v$-weighted energy estimates, and we propose some novel ideas to overcome them:

\begin{itemize}

\item \textbf{Difficulty I (error estimate):}
When there is no smallness requirement in the background profile, a more delicate obstruction occurs in the macroscopic velocity error equation $\eqref{eq:error}_4$, whose principal part is
\begin{equation}\label{partialtb1}
(a+2)\partial_t\psi-\Delta_x\psi+\partial_tb=\cdots,
\end{equation}
where $b$ is defined in \eqref{eq:def-a-b} and $\p_t b$ presents a severe obstacle as it leads to the critically indefinite term
\begin{equation*}\label{eq:bt-motivation}
\begin{aligned}
&\intx{\left(\pt \px^\a b\cdot  \px^\a \psi+\frac{| \px^\a b|^2}{\v (a+1)}+\v | \px^\a \pt \psi|^2 \right)}\\
&=\intx{\boxed{ \frac{| \px^\a b|^2}{\v (a+1)}+\v | \px^\a \pt \psi|^2- \px^\a b\cdot  \px^\a \pt\psi } }+\dt{t} \intx{ \px^\a b\cdot  \px^\a \psi},
\end{aligned}
\end{equation*}
where, if $a$ has smallness, the boxed quadratic form above is positive definite, retaining the coercivity; Without such a restriction, however, the boxed term is not guaranteed to be uniformly coercive. This loss is critical at \(|\alpha|=3\).

\textbf{Idea I:} To resolve this issue, we need extra dissipation. The key observation is that the fluid error equation expresses \(b\) as \(\varepsilon\) times a combination containing \(\partial_t\psi-\Delta_x\psi\). Differentiating this identity in time and substituting it into the preceding equation \eqref{partialtb1} reveals, at the principal level,
\[
\begin{aligned}
\varepsilon\partial_t^2\psi
-\varepsilon\Delta_x\partial_t\psi
+(a+2)\partial_t\psi
-\Delta_x\psi
&=\text{lower-order, pressure and transport terms}.
\end{aligned}
\]
This is a singularly perturbed strongly damped wave equation with Kelvin--Voigt type damping \(-\varepsilon\Delta_x\partial_t\psi\). Testing it yields the additional energy and dissipation
\[
\varepsilon\|\partial_t\psi\|_{L_t^\infty H_x^2}^2,
\qquad
\varepsilon\|\nabla_x\partial_t\psi\|_{L_t^2H_x^2}^2,
\]
respectively. The latter reaches third order derivatives and supplies exactly the dissipation missing from the direct time-differentiated estimate.

\item \textbf{Difficulty II ($\v$-weighted energy estimate):}  A standard energy argument applied directly to \((h,u)\) has difficulties in controlling the nonlinear terms, unless incorrectly treating the limiting profile itself as a small perturbation.

More precisely, the difficulty is that the variable \(\rho (=a+1)\) as in \eqref{eq:relation-g-h} need not necessarily be close to $1$. When estimating \eqref{eq:weighted-h-u} in Lemma \ref{lem:hH4}, we obtain the following energy‑dissipation structure
\begin{equation*}
\begin{aligned}
&\frac{1}{2}\dt{t}\left(\varepsilon^2\|\partial_x^\alpha h\|_{L^2_{x,v}}^2+\frac{\varepsilon^3}{2}\|\partial_x^\alpha\nabla_x u^{\v}\|_{L^2_x}^2\right)\\[5pt]
&\boxed{+\v^3\|\partial_x^\alpha \partial_t u^{\v}\|_{L^2_x}^2+\v\|\partial_x^\alpha b\|_{L^2_x}^2-\v^2\big\langle \left( \frac{\rho}{2}-1\right)\partial_x^\alpha \partial_t u^{\v},\, \partial_x^\alpha b\big\rangle_x}+\cdots.
\end{aligned}
\end{equation*}
however, since $\rho$ can be possibly large, the boxed quadratic form above are generally not positive definite.

\textbf{Idea II:} Our strategy is to eliminate $b$ using the momentum equation \eqref{nov}$_2$, and rewrite the energy‑dissipation structure as
\[
\frac{1}{2}\dt{t}\left(\varepsilon^2\|\partial_x^\alpha h\|_{L^2_{x,v}}^2 + \frac{\varepsilon^3}{2}\|\sqrt{\rho}\partial_x^\alpha\nabla_x u^{\v}\|_{L^2_x}^2\right)
+\frac{\varepsilon^3}{2}\|\sqrt{\rho}\,\partial_x^\alpha \partial_t u^{\v}\|_{L^2_x}^2+\varepsilon\|b_\alpha\|_{L^2_x}^2 + \cdots
\]
with this density‑weighted energy‑dissipation structure and the fact that $\rho$ admits a positive lower bound, we derive the desired positive‑definite structure. This strategy is essential to close the highest-order energy estimate without assuming that the high Sobolev norm of the background profile is small.
\end{itemize}

\end{itemize}

To the best of our knowledge, this is the first high-regularity justification of the fine particle VFP-INS limit that yields pointwise convergence with an explicit rate.
The proof combines a Macro--Micro decomposition around local Maxwellian with a delicate high-order uniform energy estimates, and the strategy developed here is expected to be useful for other singular limits in kinetic--fluid coupled systems where the natural equilibrium depends on macroscopic physical quantities, for example, the investigation of the Vlasov--Fokker--Planck equation coupled with the compressible Navier--Stokes equation is under preparation.
	
%------------------------------------------------------
	
\subsection{Notation and organization}

	We first present some notation that will be used throughout this paper.
	\begin{itemize}
        \item $C, C_1, C_\delta, C(T)$, and $C_\delta(T)$ denote generic constants. More specifically, $C$ and $C_1$ are independent of $\v$, $ \eta$, $T$, $T_1$, and $T^\v_{max}$; $C_\delta$ depends only on $\delta$; $C(T)>1$ only depends on $T$, and $C_\delta(T)>1$ depends on both $\delta$ and $T$.
		\item $A \le B$ denotes that there exists a generic constant $C>0$ such that $A \leq C B $.
		\item For multi-indices $\a=(\a_1,\a_2,\a_3) $ and $\b=(\b_1,\b_2,\b_3) $, we denote the binomial coefficient of $\alpha$ and $\beta$ as $C_{\alpha}^{\beta}$ and the partial derivatives
		\[
		\p^\a_{x}=\p^{\a_{1}}_{x_1}	\p^{\a_{2}}_{x_2}\p^{\a_{3}}_{x_3},\qquad \p^\b_{v}=\p^{\b_{1}}_{v_1}	\p^{\b_{2}}_{v_2}\p^{\b_{3}}_{v_3}.
		\]

		\item We denote
		\[
		\intx{g}:=\int_{\mathbb{R}^3} g \,\mathrm{d}x, \quad \intv{g}:=\int_{\mathbb{R}^3}g \,\mathrm{d}v,\quad \intxv{g}:=\int_{\mathbb{R}^3}\int_{\mathbb{R}^3} g \,\mathrm{d}x \,\mathrm{d}v,
		\]
		\[
		\|g\|_\nu:=\left(\intxv{(1+|v-u|^2)|g|^2+|\n_v g|^2}\right)^\frac{1}{2}.
		\]

		\item For an integer $k$, we denote the following spaces:
		\begin{equation*}
		\begin{aligned}
		&H^k_v:=\left\{ f\ |\ \Ltwov{\partial_v^\alpha f}< +\infty,\ \text{for}\ 0\leq |\alpha|\leq k   \right\},\\
		&\mathcal{H}^k_v:=\left\{ f\ |\ \Lv{\partial_v^\alpha f}< +\infty,\ \text{for}\ 0\leq|\alpha|\leq k   \right\},\\   	
		&H^k_x:=\left\{ u\ |\ \Ltwox{\partial_x^\alpha u}< +\infty,\ \text{for}\ 0\leq|\alpha|\leq k   \right\},\\
		&\mathcal{H}^k_x:=\left\{ u\ |\ \Ltwox{\partial_x^\alpha u}< +\infty,\ \text{for}\ 1\leq|\alpha|\leq k   \right\}.
		\end{aligned}
		\end{equation*}
	\end{itemize}
	
The rest of this paper is organized as follows. In Section \ref{subsec:main-thoerem}, we introduce the Macro-Micro decomposition around the local Maxwellian, reformulate the original system into the coupled macroscopic and microscopic equations, and derive the governing system for the macroscopic error terms. Section \ref{sec:global} is devoted to establishing the uniform \textit{a priori} energy estimates with respect to $\varepsilon$, which consist of the macroscopic error estimates and the high-order estimates of the microscopic parts. Finally, some preliminary lemmas and necessary technical tools are collected in the Appendix.

%%%%%%%%%%%%%%%%%%%%%%%%%%%%%%%%%%%%%%%

\section{Main results}
\label{subsec:main-thoerem}

In this section, we will present our basic setup and main results.

%-------------------------------
\subsection{Macro-Micro decomposition}
\label{Macro-Micro decomposition}
	
We consider the Macro-Micro decomposition of $f^\v$ in the sense that
	\begin{equation}\label{MM}
	f^\varepsilon(t,x,v)=M_{u^\v (t,x)}(v)+g^{\v}(t,x,v)\sqrt{M_{u^\v (t,x)}(v)},
	\end{equation}
where
	\begin{equation}\label{Mlocal}
	M_{u^\v (t,x)}(v)=\frac{1}{(2\pi)^\frac{3}{2}} \mathrm{exp}\left\{-\frac{|v-u^\v(t,x)|^2}{2}\right\}.
	\end{equation}
For simplicity, we omit the $\varepsilon$-dependence in notation, i.e., denoting
\begin{equation}\label{nov}
    M_u:=M_{u^\v}, \quad g:=g^\v, \quad u:=u^\v, \quad P:=P^\v,
\end{equation}
such that the VFP-INS system \eqref{f} is rewritten as
\begin{equation}\label{The system of (g,u)}
\left\{
\begin{aligned}
&\p_t g + v\cdot \n_x g + \frac{1}{\v}\mathbf{L} g = R_0,\\[5pt]
&\p_t u + u\cdot \n_x u + \n_x P-\Delta_x u =\frac{b}{\v},\\[5pt]
&\big( g(t,x,v),u(t,x) \big)|_{t=0} = \big( g^{\v,in}(x,v),u^{\v,in}(x) \big),
\end{aligned}
\right.
\end{equation}
where the linearized operator $\mathbf{L}$ and the remainder term $R_0$ are given by
\begin{equation}\label{sourse}
\begin{aligned}
&\mathbf{L}g := -\frac{1}{\M}\mathrm{div}_v\left[ M_u \n_v \left(\frac{g}{\M}\right) \right], \\[5pt]
&R_0:=-(v-u)\cdot \p_t u \M- \big[(v\cdot \n_x) u\big]\cdot (v-u)\M - \frac{1}{2}g(v-u)\cdot \p_t u\\
& \ \qquad-\frac{1}{2}g\big[(v\cdot \n_x)u\big]\cdot (v-u).
\end{aligned}
\end{equation}
The operator $\mathbf{L}$ satisfies a dissipation property (see Lemma \ref{lemma A2}):
\begin{equation*}
c\|(\mathbf{I-P})g\|_{\nu}^2+\Lx{b}^2 \leq \intxv {g\mathbf{L}g}.
\end{equation*}

We can define the macroscopic quantities $a$ and $b$ as
\begin{equation}\label{eq:def-a-b}
a: = a(t,x) = \int  g\M \,\d v, \quad b: = b(t,x) = \int (v-u)g\M \,\d v,
\end{equation}
and then decompose $g$ into the macroscopic part $\mathbf{P}g $ and the microscopic part $(\mathbf{I-P})g $ as follows:
	\begin{equation}\label{decom}
	g=\mathbf{P}g+(\mathbf{I-P})g,
	\end{equation}
	where the projection $\mathbf{P}$ is given by
	\begin{equation*}
	\mathbf{P}=\mathbf{P}_0+\mathbf{P}_1,\quad \text{with} \quad  \mathbf{P}_0g :=a\M, \quad  \mathbf{P}_1 g :=b\cdot (v-u)\M.
	\end{equation*}

	To close the system \eqref{The system of (g,u)}, we need to derive the macroscopic system of $a$ and $b$: multiplying $\eqref{The system of (g,u)}_1$ by $\M $ and $(v-u)\M $, and then integrating over $v \in \mathbb{R}^3 $, we obtain
	\begin{equation}\label{g-0}
	\left\{
	\begin{aligned}
	&\p_t a +\up{div}_x b+u\c\n_x a =   0,\\[5pt]
	&\p_t b + \p_t u + \n_x a + \frac{b}{\v} = R_1,
	\end{aligned}
	\right.
	\end{equation}
	where the remainder term $R_1$ reads
	\begin{equation}\label{sourse-1}
	\begin{aligned}
	R_1 :=& -a \p_t u - au\cdot \n_xu -u\c\n_x u - u\cdot \n_x b - b\cdot \n_xu\\[3pt]
	&- \l v\cdot\n_x((\mathbf{I-P})g\M),v\r_v.
	\end{aligned}
	\end{equation}
    To eliminate the singular term $\frac{b}{\v}$ in $\eqref{The system of (g,u)}_2$, we add this equation to $\eqref{g-0}_2$ and arrive at:
	\begin{equation}\label{The equation of ptu+ptb}
	\begin{array}{ll}
	2\p_t u + \p_t b +u\cdot \n_x u+ \n_x a + \n_x P-\Delta_xu = R_1.
	\end{array}
	\end{equation}
	By further combining \eqref{g-0} and \eqref{The equation of ptu+ptb}, we obtain
	\begin{equation}\label{g-1}
	\left\{
	\begin{aligned}
	&\p_t a +\up{div}_x b +u\c\n_x a= 0,\\[5pt]
	&\p_t b + \p_t u + \n_x a + \frac{b}{\v} = R_1,\\[5pt]
	&2\p_t u + \p_t b +u\cdot \n_xu+ \n_x (a+P)-\Delta_xu = R_1.
	\end{aligned}
	\right.
	\end{equation}

Finally, gathering \eqref{The system of (g,u)} and \eqref{g-1} yields the following scaled VFP-INS system based on the Macro-Micro decomposition \eqref{MM}:
	\begin{equation}\label{g}
	\left\{
	\begin{aligned}
	&\p_t g + v\cdot \n_x g + \frac{1}{\v}\mathbf{L} g = R_0,\\[3pt]
	&\p_t u + u\cdot \n_x u + \n_x P-\Delta_x u = \frac{b}{\v},\\[3pt]
	&\p_t a +\up{div}_x b +u\c\n_x a= 0,\\[3pt]
	&\p_t b + \p_t u + \n_x a + \frac{b}{\v} = R_1,\\[3pt]
	&2\p_t u + \p_t b + u\cdot \n_x u+ \n_x (a+P)-\Delta_xu = R_1,\\
	\end{aligned}
	\right.
	\end{equation}
	with $\mathbf{L} g,\,R_0,\,R_1$ given in \eqref{sourse} and \eqref{sourse-1}, and the initial conditions
	\begin{equation*}
	\left\{
	\begin{aligned}
	&\big( g(t,x,v),u(t,x) \big) |_{t=0} = \big( g^{\v,in}(x,v),u^{\v,in}(x) \big),\\[3pt]
	&a(t,x)|_{t=0} = a^{in}(x):=\intv{g^{\v,in}\sqrt{M^{in}}},\\[3pt]
	&b(t,x)|_{t=0} = b^{in}(x):=\intv{(v-u^{\v,in})g^{\v,in}\sqrt{M^{in}}},
	\end{aligned}
	\right.
	\end{equation*}
	where the initial Maxwellian distribution $M^{in}$ is
	\begin{equation*}\label{eq:Min}
	M^{in}:=\frac{1}{(2\pi)^{3/2}} \mathrm{exp}\left\{ -\frac{|v-u^{\v,in}(x)|^2}{2}\right\}.
	\end{equation*}

To clarify the limiting process from the VFP-INS system \eqref{f} to the IINS system \eqref{limit}, we introduce the macroscopic error functions defined by
\begin{equation}\label{ma-error}
\phi:=a-\rho_0,\quad \psi:=u-u_0,
\end{equation}
where $(a,u)$ and $(\rho_0,u_0)$ denote the solutions to \eqref{g} and \eqref{limit}, respectively. Consequently, the governing system for the error terms $(\phi,\psi)$ can be derived as follows:
\begin{equation}\label{eq:error}
\left\{
\begin{aligned}
&\p_t \phi+\div_x b+\n_x\phi\cdot\psi+\n_x\rho_0\cdot\psi+\n_x\phi\cdot u_0 =0,\\[5pt]
&(a+1)\p_t \psi+(a+1)u\cdot\n_x\psi+\n_x\phi+\p_t b+\frac{b}{\v} =-R_2-R_3,\\[5pt]
&\p_t \psi+u\cdot\n_x\psi+\n_x(P+\rho_0-P_0)-\Delta_x\psi-\frac{b}{\v} =R_3-R_4,\\[5pt]
&(a+2)\p_t \psi+(a+2)u\cdot\n_x\psi+\n_x(P+a-P_0)-\Delta_x\psi+\p_t b =-R_2-R_4,\\[5pt]
&\div_x\psi =0,
\end{aligned}
\right.
\end{equation}
with the initial condition $(\phi,\psi)|_{t=0}=(a^{in}-\rho_0^{in},u^{in}-u^{in}_0)$, where the remainder terms $R_2, R_3$, and $R_4$ are given by
\begin{equation}\label{eq:sourse 3}
\left\{
\begin{aligned}
R_2 =& \ \phi \p_t u_0+(a+1)\psi\cdot\n_xu_0+\phi u_0\cdot\n_xu_0+u\cdot \n_x b\\[3pt]
& +b\cdot\n_xu+\l v\cdot\n_x((\mathbf{I-P})g\M),v\r_v,\\[5pt]
R_3 =&\  (\rho_0+1)\p_t u_0+(\rho_0+1)u_0\cdot\n_xu_{0}+\n_x\rho_0,\\[5pt]
R_4 =&\  \psi\cdot\n_xu_0.
\end{aligned}
\right.
\end{equation}

%--------------------
\subsection{Main theorem}

In this subsection, we present our main theorem. For clarification, we first define the total energy functional $\E(t)$ as
\begin{equation}\label{defE}
\begin{aligned}
\E(t)=&\Hxxx{(\phi,\psi,b)}^2+\frac{\Hxx{b}^2}{\varepsilon}+\varepsilon\Hxx{\p_t\psi}^2+\varepsilon\Hxxx{\nabla_x\psi}^2+\varepsilon^2\HxxxLv{\nabla_xh}^2+\varepsilon^3\Hthreex{\nabla_x^2 u}^2\\[5pt]
&+\|(\mathbf{I-P})g\|_{H^2_x H^2_v}^2+\v\|(\mathbf{I-P})g\|_{H^3_x L^2_v}^2,
\end{aligned}
\end{equation}
where $h=g+\sqrt{M_u} $.

As noted above, our main hydrodynamic limit theorem holds provided the existence of a sufficiently regular limiting profile on a prescribed finite time interval $[0,T]$.
Hence, we present the well-posedness and regularity theory of the limiting IINS system \eqref{limit} to ensure the existence of such a case: for certain initial conditions, there exists a solution $(\rho_0,u_0)$ to the limiting system \eqref{limit} satisfying
\begin{equation}\label{eq:regu limit}
0\leq \rho_0\leq \overline{\rho}<+\infty,\quad\partial^k_t \rho_0, \ \partial^k_t u_0 \in  C\big([0,T];H^{6-2k}_x(\mathbb{R}^3)\big)
\end{equation}
with $k=0,1,2$.

This result therefore provides a non-empty class of limiting profiles to which our main theorem applies. We refer the reader to Proposition \ref{prop:limit} and Corollary \ref{coro:limit} for further details.

We are now in a position to present our main theorem.

\begin{theorem}\label{thm1}
	Assume the initial condition of the system \eqref{f} satisfies
    \[f^{\varepsilon,in} = M^{in}+g^{in}\sqrt{M^{in}}, \quad u^{\varepsilon,in} = u^{in}
    \]
    with
	\[
	g^{in}\in H^4_xL^2_v\cap H^2_xH^2_v,\quad u^{in}\in H^5_x.
	\]
	and the initial condition of the systems \eqref{g} and \eqref{eq:error} satisfies the well-preparedness condition, for $0<\gamma\leq 1$,
	\begin{equation}\label{The condition of E}
	\E(0)\leq C\v^\gamma.
	\end{equation}
	Then, for any $T>0$, there exists $\varepsilon_0=\varepsilon_0(T)>0$ such that, for any $0<\varepsilon\leq\varepsilon_0$,  the system \eqref{g} admits a unique solution $(g,u)$ on $[0,T]$
    satisfying
	\begin{equation*}\label{E(t)}
	\sup_{t\in[0,T]}\mathbb E(t)\leq C(T)(\v^\gamma+\varepsilon).
	\end{equation*}
	In particular, the following convergence estimate
    \begin{equation*}\label{eq:thm1-conv}
    \begin{aligned}
    &\sup_{t\in[0,T]}
	\Big(
	\|f^\varepsilon(t,x,v)-(1+\rho_0(t,x))M_{u_0(t,x)}(v)\|^2_{L_x^\infty L_v^\infty}
	+\|u^\varepsilon(t,x)-u_0(t,x)\|^2_{L_x^\infty}
	\Big)\\[3pt]
	\leq &C(T)(\v^\gamma+\varepsilon)
    \end{aligned}
    \end{equation*}
    holds,
	where $(f^\v, u^\v)$ in the form of \eqref{MM}-\eqref{nov} is the solution to  system \eqref{f}, and $(\rho_0,u_0)$ is the solution to system \eqref{limit} satisfying \eqref{eq:regu limit}.
\end{theorem}

\begin{remark}
The possibly large limiting profiles and compatibility of the initial energy are further explained as follows:
\begin{enumerate}
\item The well-preparedness condition \eqref{The condition of E} is imposed on the error and non-equilibrium components, rather than on the leading macroscopic profile itself. Indeed, the full macroscopic variables \((a,u)\) need not be small, since
\[
a=\rho_0+\phi,
\qquad
u=u_0+\psi.
\]
%with the limiting profile \((\rho_0,u_0)\).
Only the macroscopic errors \((\phi,\psi)\), the rapidly relaxing momentum \(b\), and the microscopic component \(\mi g\) are inherently expected to vanish $\varepsilon \to 0$.\\
Also note that, the $\v$-weights in \eqref{defE} are essential here. Suppose, for example, that
\[
\HxxxLv{\nabla_x h^{in}}^2
+
\Hthreex{\nabla_x^2u^{in}}^2
\leq C_{in},
\]
where \(C_{in}\) is independent of \(\varepsilon\), but is not required to be small and may depend on the size of the macroscopic profile. Then, for \(0<\varepsilon \leq 1\) and \(0<\gamma \leq 1\),
\[
\varepsilon^2\HxxxLv{\nabla_x h^{in}}^2
+
\varepsilon^3\Hthreex{\nabla_x^2u^{in}}^2
\leq C_{in}\varepsilon^2
\leq C_{in}\varepsilon^\gamma.
\]
Thus, the highest-order unweighted norms of the full kinetic and fluid variables may be large, while their contributions to the initial energy remain compatible with \eqref{The condition of E}.

\item In fact, the well-prepared requirement \eqref{The condition of E}  can be removed for a certain class of initial conditions. For instance, if the initial condition of \eqref{f} is exactly chosen as the local Maxwellian type as below:
\[
f^{\varepsilon, in}
=
(1+\rho_0^{in})M_{u_0^{in}},
\qquad
u^{\varepsilon,in}=u_0^{in},
\]
it directly implies
\[
\phi^{in}=\psi^{in}=b^{in}
=\mi g^{in}=0.
\]
Suppose the remaining time-derivative and weighted higher-order quantities are uniformly bounded, this gives \(\E(0)=O(\varepsilon)\), and hence satisfies \eqref{The condition of E}. 
Therefore, the initial requirement \eqref{The condition of E} is naturally satisfied, while allowing a possibly large, sufficiently regular macroscopic background profile.
\end{enumerate}
\end{remark}

%%%%%%%%%%%%%%%%%%%%%%%%%%%%%%%%%
\section{Proof of Theorem \ref{thm1}}
\label{sec:global}

This section is devoted to the proof of Theorem \ref{thm1}. To achieve this, we first present the local well-posedness of the system \eqref{g} in subsection \ref{subsec:local}. In subsection \ref{subsection:error}, our primary task is to derive the corresponding energy estimates. Subsequently, we will establish the bounds for the kinetic perturbation $g$ and its microscopic projection $(\mathbf{I-P})g$ in subsection \ref{subsec:micro-estimate}. Finally, by synthesizing these estimates, we obtain the uniform total energy bound required to close the proof of Theorem \ref{thm1} in subsection \ref{subsec:proof thm1}.

%----------------------
\subsection{Local well-posedness}
\label{subsec:local}

First, we give the local existence of the system \eqref{g}.

\begin{proposition}[Local well-posedness]\label{prop:local_existence}
	Under the conditions of Theorem \ref{thm1}, there exists a time $T^\varepsilon > 0$, depending only on $\varepsilon$ and the initial data, such that system \eqref{g} admits a unique strong solution $(g,u)$ on $[0, T^\varepsilon]$ with
	\begin{equation*}
	\begin{aligned}
	&g\in C([0,T^\varepsilon];H^4_xL^2_v\cap H^2_xH^2_v)\cap L^2(0,T^\varepsilon;H^4_x\mathcal{H}^0_v\cap H^2_x\mathcal{H}^2_v),\\
	&u\in C([0,T^\varepsilon];H^5_x)\cap L^2(0,T^\varepsilon;\mathcal{H}^5_x),\ \partial_tu\in C([0,T^\varepsilon];H^2_x)\cap L^2(0,T^\varepsilon;H^4_x).
	\end{aligned}
	\end{equation*}
\end{proposition}

\begin{remark}
The local well-posedness argument is a standard adaptation of the construction developed for closely related Vlasov--Fokker--Planck--Navier--Stokes systems in \cite{Goudon2010}.
Indeed, for each fixed $\varepsilon>0$, the coefficients of order $\varepsilon^{-1}$ are finite and therefore cause no additional difficulty in the local construction, although the resulting lifespan and estimates may depend on $\varepsilon$.

\end{remark}

%-------------------------
\subsection{Macroscopic error estimates}
\label{subsection:error}

The key to proving Theorem \ref{thm1} is to obtain uniform energy estimates based on the following \textit{a priori} assumption:
\begin{equation}\label{assumption}
\sup_{t\in[0,T_1]}\mathbb{E}(t)\leq 2 \eta^2,
\end{equation}
where $\mathbb{E}(t)$ is the energy functional defined in \eqref{defE}, and $\eta>0$ is a sufficiently small constant independent of $\varepsilon$ to be determined in \eqref{eq:eta} and \eqref{eq:def eta}. Noting that $\E(0)\leq C\v ^\gamma$ and choosing
\begin{equation}\label{eq:varepsilon}
    C\varepsilon^\gamma\leq \eta^2,
\end{equation}
we have $T_1\in (0,T^{\v}_{max})$, where $ T^{\v}_{max}$ is the maximal existence time in Proposition \ref{prop:local_existence}.

In our proof, we repeatedly use the existence of solutions to the limit equation \eqref{limit} on $[0,T]$. We can therefore assume \(T_1\leq T\) and aim to apply the standard continuity argument to establish that Theorem \ref{thm1} holds on $[0,T]$.

Under the \textit{a priori} assumption \eqref{assumption}, we can derive the energy estimates for the macroscopic error terms.

\begin{proposition}\label{lem:ma g}	Under the conditions of Theorem \ref{thm1} and the \textit{a priori} assumption \eqref{assumption}, for $t\in [0,T_1] $, we have
\begin{equation}\label{eq:ma g}
        \begin{aligned}
        &\Hxxx{(\phi,\psi,b)}^2+\frac{\Hxx{b}^2}{\varepsilon}+\varepsilon\Hxx{\p_t\psi}^2+\varepsilon\Hxxx{\nabla_x\psi}^2+\varepsilon^2\HxxxLv{\nabla_xh}^2+\varepsilon^3\Hthreex{\nabla_x^2 u}^2\\
        &+\int_0^t\left( \Hxxx{\nabla_x\psi}^2+\varepsilon\Hxxxx{\nabla_x\psi}^2+\frac{\Hxxx{b}^2}{\varepsilon}+\Hxx{\p_tb}^2+\Hxx{\p_t\psi}^2+\varepsilon\Hxx{\nabla_x\p_t\psi}^2\right)\d s\\
        &+\intt{\left(\varepsilon^3\Hthreex{\nabla_x\pt u}^2+\sum\limits_{1\leq|\alpha|\leq4}\varepsilon(\Ltwox{b_\alpha}^2+\Lv{\mi \px^\alpha h}^2)\right)}{s}\\
     \leq&C(T)\E(0)+C(T)\varepsilon+C(T)\intt{\E(s)}{s}+C(T)\intt{\sum\limits_{|\theta|\leq3}\Lv{\px^\theta\mi g}^2}{s}\\
        \end{aligned}
        \end{equation}
         where the energy functional $\E(t)$ is given in \eqref{defE} and $C(T) > 1$ is a constant only depending on $T$.
	\end{proposition}
The estimate in Proposition \ref{lem:ma g} is a direct combination of Lemmas \ref{lem:phiH3}--\ref{lem:hH4} and Corollaries \ref{lem:phi b psi psit}--\ref{coro:ma total}. In what follows, we will present these lemmas and corollaries with proofs, and provide the proof of Proposition \ref{lem:ma g} at the end of this subsection.
%%%%%%%%%%%%%%%%%%%%%%%%%%%%%%%%%%%%%%%%%%%%	
\begin{lemma}\label{lem:phiH3}
  Under the conditions of Theorem \ref{thm1} and the \textit{a priori} assumption \eqref{assumption}, for $t\in [0,T_1] $, we have
  	\begin{equation}\label{eq:phiH3}
  	\begin{aligned}
  	\sum\limits_{ |\a|\leq 3}\left( \frac{1}{2}\dt{t}\intx{\frac{|\p_x^\a \phi|^2}{a+1}}+\intx{\frac{\p_x^\a\div_xb\,\p_x^\a\phi}{a+1}} \right)\leq C(T)\E(t),
  	\end{aligned}
  	\end{equation}
    where the energy functional $\E(t)$ is given in \eqref{defE} and $C(T) > 1$ is a constant only depending on $T$.
  \end{lemma}

  \begin{proof}	
  	Applying $\partial_{x}^{\alpha}$ ($|\alpha|\leq 3$) to  $\eqref{eq:error}_1$, multiplying by $\frac{\partial_{x}^{\alpha}\phi}{a+1}$, and integrating by parts, we obtain
  	\begin{equation}\label{eq:lem phi}
  	\begin{aligned}
  	&\frac{1}{2}\dt{t}\intx{\frac{|\p_x^\a \phi|^2}{a+1}}+\intx{\frac{\p_x^\a\div_xb\,\p_x^\a\phi}{a+1}}\\[5pt]
  	=&\underbrace{-\intx{\frac{|\p_x^\a \phi|^2\p_ta}{(a+1)^2}}}_{I_{11}}\underbrace{-\intx{\frac{\p_x^\a(\nabla_x\phi\cdot\psi)\p_x^\a \phi}{a+1}}}_{I_{12}}\underbrace{-\intx{\frac{\p_x^\a(\nabla_x\rho_0\cdot\psi)\p_x^\a \phi}{a+1}}}_{I_{13}}\\[5pt]
  	&\underbrace{-\intx{\frac{\p_x^\a(\nabla_x\phi\cdot u_0)\p_x^\a \phi}{a+1}}}_{I_{14}}.
  	\end{aligned}
  	\end{equation}  	
  	For $I_{11}$, substituting $\eqref{g}_3$ and considering $\phi=a-\rho_0,\,\psi=u-u_0$, we have
  	\begin{equation}\label{eq:a_t Linfty}
  	\begin{aligned}
  	\Linf{\p_ta}\leq& \Linf{\div_xb}+\Linf{u\cdot\nabla_x a}\\[5pt]
  	\leq& C\Hxxx{b}+\Hxx{(\psi+u_0)\cdot\nabla_x(\phi+\rho_0)}\\[5pt]
  	\leq &C(T)+C\E(t).
  	\end{aligned}
  	\end{equation}	

    Since $a=\phi+\rho_0$, we can combine $0 < 1 \leq \rho_0 +1 \leq \overline{\rho}+1$ (See \eqref{eq:regu limit}) with the \textit{a priori} assumption \eqref{assumption} to obtain
  	\begin{equation}\label{eq:a lower bound}
  	a+1=\phi+\rho_0+1\in \left(\frac{1}{2}, \overline{\rho}+\frac{3}{2}\right),
  	\end{equation}
    where, we select \(\eta>0\) sufficiently small such that
    \begin{equation}\label{eq:eta}
        \Linftyx{\phi}\leq C\Hxx{\phi}\leq C\sqrt{\E(t)}\leq C\eta \leq \frac{1}{4}.
    \end{equation}
  	Then, we have
  	\begin{equation}\label{eq:I11}
  	I_{11}=-\intx{\frac{|\p_x^\a \phi|^2\p_ta}{(a+1)^2}}\leq  C\Linftyx{\pt a}\Ltwox{\px^\a \phi}^2\leq C(T)\E(t).
  	\end{equation}

  	For $I_{12}$, integrating by parts yields
  	\begin{equation}\label{eq:I12}
  	\begin{aligned}
  	I_{12}=&-\intx{\frac{\p_x^\a(\nabla_x\phi\cdot\psi)\p_x^\a \phi}{a+1}}\\[5pt]
  	=&-\intx{\frac{\p_x^\a\nabla_x\phi\cdot\psi\,\p_x^\a \phi}{a+1}}-\sum\limits_{\b\leq\a,|\b|\neq 0}C^\b_\a\intx{\frac{\p_x^{\a-\b}\nabla_x\phi\cdot\p_x^\b\psi\,\p_x^\a \phi}{a+1}}\\[5pt]
  	=&\intx{\frac{\div_x\psi |\p_x^\a \phi|^2}{2(a+1)}}-\intx{\frac{\psi\cdot \nabla_x a |\p_x^\a \phi|^2}{2(a+1)^2}}-\sum\limits_{\b\leq\a,|\b|\neq 0}C^\b_\a\intx{\frac{\p_x^{\a-\b}\nabla_x\phi\cdot\p_x^\b\psi\,\p_x^\a \phi}{a+1}}\\[5pt]
    \leq&C(\Hthreex{\psi}+\Hthreex{\psi}\Linftyx{\nabla_x a})\Hthreex{\phi}^2\\[5pt]
  	\leq& C(T)\E(t),
  	\end{aligned}
  	\end{equation}
  	where we use \eqref{assumption} and the fact that
    \begin{equation*}
  	\begin{aligned}
  	\Linf{\nabla_x a}\leq \Linf{\nabla_x\rho_0+\nabla_x\phi}\leq C(T)+C\E(t).
  	\end{aligned}
  	\end{equation*}
  	Similarly, for $I_{13}$ and $I_{14}$, we have
  	\begin{equation}\label{eq:I13 14}
  	\begin{aligned}
  	I_{13}+I_{14}=&-\intx{\frac{\p_x^\a(\nabla_x\rho_0\cdot\psi)\p_x^\a \phi}{a+1}}-\intx{\frac{\p_x^\a(\nabla_x\phi\cdot u_0)\p_x^\a \phi}{a+1}}\\[5pt]
    \leq&C\Hfourx{\rho_0}\Hthreex{\psi}\Hthreex{\phi}+C\Hthreex{u_0}\Hthreex{\phi}^2\\[5pt]
  	\leq& C(T)\E(t).
  	\end{aligned}
  	\end{equation}
  	
  	Finally, collecting the estimates \eqref{eq:lem phi}, \eqref{eq:I11}, \eqref{eq:I12}, \eqref{eq:I13 14}, and summing over $|\alpha| \leq 3$, we obtain the estimate \eqref{eq:phiH3}.
  \end{proof}

Having established the energy inequality for $\phi$, we now proceed to derive the energy estimate for $\psi$.
	\begin{lemma}\label{lem:psiH3}
    Under the conditions of Theorem \ref{thm1} and the \textit{a priori} assumption \eqref{assumption}, for $t\in [0,T_1] $, we have
		\begin{equation}\label{eq:psiH3}
		\begin{aligned}
		&\sum\limits_{|\a|\leq 3}\left( \frac{1}{2}\dt{t}\intx{(a+2)|\p_x^\a\psi|^2}+\intx{|\n_x\p_x^\a\psi|^2} \right) \\
		&+\sum\limits_{|\a|\leq 3}\left( \varepsilon\dt{t}\intx{\p_x^\a\p_t\psi\cdot\p_x^\a\psi}-\varepsilon\intx{|\p_x^\a\p_t\psi|^2}+\frac{\varepsilon}{2}\dt{t}\intx{|\p_x^\a\nabla_x\psi|^2}\right) \\[5pt]
		\leq&\ \delta\Hxxx{\nabla_x\psi}^2+\delta\Hxx{\p_t\psi}^2+C_\delta(T)\E(t)+C_\delta(T)\Hxxxv{\mi g}^2\\[5pt]
		&+C_\delta(T)\varepsilon^2\Hxxx{\p_t\psi}^2+C_\delta(T)\varepsilon^2,
		\end{aligned}
		\end{equation}
        where the energy functional $\E(t)$ is given in \eqref{defE} and $C_\delta(T) > 1$ is a constant only depending on $\delta,\ T$.
	\end{lemma}
	\begin{proof}	
    Applying $\partial_{x}^{\alpha}$ ($|\alpha|\leq 3$) to $\eqref{eq:error}_4$, multiplying by $\partial_{x}^{\alpha}\psi$, and integrating by parts, we obtain
		\begin{equation}\label{eq:psiH3-1}
		\begin{aligned}
		&\frac{1}{2}\dt{t}\intx{(a+2)|\p_x^\a\psi|^2}+\intx{|\n_x\p_x^\a\psi|^2}+\underbrace{\intx{\p_x^\a\p_t b\cdot\p_x^\a\psi}}_{I_{20}} \\
		=&\underbrace{\frac{1}{2}\intx{\p_t a|\p_x^\a\psi|^2}}_{I_{21}}\underbrace{-\sum\limits_{\b\leq\a,|\b|\neq 0}C^\b_\a\intx{\p_x^\b a\,\p_x^{\a-\b}\p_t\psi\cdot\p_x^\a\psi}}_{I_{22}} \\
		&\underbrace{-\intx{\p_x^\a\big((a+2)u\cdot\n_x\psi\big)\cdot\p_x^\a\psi}}_{I_{23}}\underbrace{-\intx{\p_x^\a R_2\cdot\p_x^\a\psi}}_{I_{24}}\underbrace{-\intx{\p_x^\a R_4\cdot\p_x^\a\psi}}_{I_{25}}.
		\end{aligned}
		\end{equation}
		To estimate $I_{20}$, we use $\eqref{eq:error}_3$ to write $\p_t b$ as:
		\begin{equation}\label{eq:bt}
		\p_t b=\varepsilon \p_{tt}\psi-\varepsilon\Delta_x\p_t\psi+\varepsilon\p_t(u\cdot\nabla_x\psi)+\varepsilon\p_t\nabla_x(P+\rho_0-P_0)-\varepsilon\p_tR_3+\varepsilon\p_tR_4.
		\end{equation}
		Substituting \eqref{eq:bt} into $\intx{\p_x^\a\p_t b\cdot\p_x^\a\psi}$ yields
		\begin{equation}\label{eq:psi bt-1}
		\begin{aligned}
		&\intx{\p_x^\a\p_t b\cdot\p_x^\a\psi} \\
		=&\intx{\p_x^\a\psi\cdot \p_x^\a\left(\varepsilon \p_{tt}\psi-\varepsilon\Delta_x\p_t\psi+\varepsilon\p_t(u\cdot\nabla_x\psi)-\varepsilon\p_tR_3+\varepsilon\p_tR_4 \right) } \\
		=&\varepsilon\dt{t}\intx{\p_x^\a\p_t\psi\cdot\p_x^\a\psi}-\varepsilon\intx{|\p_x^\a\p_t\psi|^2}+\frac{\v}{2}\dt{t}\intx{|\p_x^\a\nabla_x\psi|^2}\\
		&+\underbrace{\varepsilon\intx{\p_x^\a\psi\cdot \p_x^\a\p_t(u\cdot\nabla_x\psi) }}_{N_{11}}+\underbrace{\varepsilon\intx{\p_x^\a\psi\cdot \p_x^\a\left(-\p_tR_3+\p_tR_4 \right) }}_{N_{12}}.
		\end{aligned}
		\end{equation}
		For $N_{11}$, integrating by parts and using $u=\psi+u_0$, we obtain, for $|\a'|=1$
		\begin{equation}\label{eq:J11}
		\begin{aligned}
		\left|N_{11}\right|=&\left|\varepsilon\intx{\p_x^\a\psi\cdot \p_x^\a\p_t(u\cdot\nabla_x\psi) }\right|\\
		\leq&\left|\varepsilon\intx{\psi\cdot \p_t(u\cdot\nabla_x\psi) }\right|+\sum\limits_{1\leq |\a|\leq 3}\left|\varepsilon\intx{\p_x^{\a+\a'}\psi\cdot \p_x^{\a-\a'}\p_t(u\cdot\nabla_x\psi) }\right|\\
        \leq &C\v\Hthreex{\psi}\Hthreex{\nabla_x \psi}\Big(\Hthreex{\pt\psi}+\Hthreex{\pt u_0}\Big)\\
        &\qquad \qquad\qquad\qquad\qquad \qquad+C\v \Hthreex{\pt\psi}\Big(\Hthreex{\psi}+\Hthreex{u_0}\Big) \Hthreex{\nabla\psi}  \\[5pt]
		\leq&\ C_\delta(T)\E(t)+\delta\Hxxx{\nabla_x\psi}^2+C_\delta\varepsilon^2(C(T)+\Hxxx{\psi}^2)\Hxxx{\p_t\psi}^2\\[5pt]
		\leq&\ \delta\Hxxx{\nabla_x\psi}^2+C_\delta(T)\E(t)+C_\delta(T)\varepsilon^2\Hxxx{\p_t\psi}^2.
		\end{aligned}
		\end{equation}
		For $N_{12}$, integrating by parts, we have
		\begin{equation}\label{eq:J12}
		\begin{aligned}
		\left|N_{12}\right|=&\left|\varepsilon\intx{\p_x^\a\psi\cdot \p_x^\a\left(-\p_tR_3+\p_tR_4 \right) }\right|\\[5pt]
		=&\left|\varepsilon\intx{\p_x^\a\psi\cdot \p_x^\a\p_t
        \left( (\rho_0+1)\p_t u_0+(\rho_0+1)u_0\cdot\n_xu_{0}+\n_x\rho_0-\psi\cdot\nabla_xu_0 \right) }\right|\\[5pt]
		\leq&\ C(T)\E(t)+C_\delta(T)\varepsilon^2+\delta\Hxxx{\nabla_x\psi}^2+C_\delta(T)\v^2\Hthreex{\pt \psi}^2,
		\end{aligned}
		\end{equation}
        where we have used
        \begin{equation*}
        \begin{aligned}
        &\left|\v\intx{\p_x^\a\psi\cdot \p_x^\a\p_t\left(\psi\cdot\nabla_xu_0 \right) }\right|\\
        =&\left|\v\intx{\p_x^{\a}\nabla_x\psi:\p_x^\a\p_t\left(\psi\otimes u_0 \right) }\right|\\[5pt]
        \leq&\delta \Hthreex{\n_x\psi}^2+C_\delta\v^2\Hthreex{\psi}^2\Hthreex{\pt u_0}^2+C_\delta\v^2\Hthreex{u_0}^2\Hthreex{\pt \psi}^2\\[5pt]
        \leq& \delta \Hthreex{\n_x\psi}^2+C_\delta(T)\v^2+C_\delta(T)\v^2\Hthreex{\pt \psi}^2.
                \end{aligned}
        \end{equation*}
		Therefore, combining \eqref{eq:psi bt-1}, \eqref{eq:J11}, and \eqref{eq:J12}, we get
		\begin{equation}\label{eq:psi bt}
		\begin{aligned}
		\intx{\p_x^\a\p_t b\cdot\p_x^\a\psi}
		\geq&\ \varepsilon\dt{t}\intx{\p_x^\a\p_t\psi\cdot\p_x^\a\psi}-\varepsilon\intx{|\p_x^\a\p_t\psi|^2}+\frac{\v}{2}\dt{t}\intx{|\p_x^\a\nabla_x\psi|^2}\\[5pt]
		&-\left(\delta\Hxxx{\nabla_x\psi}^2+C_\delta (T)\E(t)+C_\delta(T)\varepsilon^2\Hxxx{\p_t\psi}^2+C_\delta (T)\varepsilon^2 \right).
		\end{aligned}
		\end{equation}
		
		Next, we estimate the terms on the right-hand side of \eqref{eq:psiH3-1}. For $I_{21}$, using \eqref{eq:a_t Linfty},
		\begin{equation}\label{eq:I21}
		I_{21}=\frac{1}{2}\intx{\p_t a|\p_x^\a\psi|^2}\leq C\Linftyx{ \pt a}\Hthreex{\psi}^2\leq C(T)\E(t).
		\end{equation}	
		For $I_{22}$, noting that $a=\phi+\rho_0$, we have	
		\begin{equation}\label{eq:I22}
		\begin{aligned}
		I_{22}=&-\sum\limits_{\b\leq\a,|\b|\neq 0}C^\b_\a\intx{\p_x^\b a\,\p_x^{\a-\b}\p_t\psi\cdot\p_x^\a\psi}\\
		=&-\sum\limits_{\b\leq\a,|\b|\neq 0}C^\b_\a\intx{\p_x^\b (\phi+\rho_0)\,\p_x^{\a-\b}\p_t\psi\cdot\p_x^\a\psi}\\[5pt]
		\leq&\ C(\Hxxx{\phi}+\Hxxx{\rho_0})\Hxx{\p_t\psi}\Hxxx{\psi}\\[5pt]
		\leq&\ C_\delta(T)\E(t)+\delta\Hxx{\p_t\psi}^2.
		\end{aligned}
		\end{equation}
		For $I_{23}$, we have
		\begin{equation}\label{eq:I23}
		\begin{aligned}
		I_{23}=&-\intx{\p_x^\a\big((a+2)u\cdot\n_x\psi\big)\cdot\p_x^\a\psi}\leq \delta\Hxxx{\nabla_x\psi}^2+C_\delta (T)\E(t).
		\end{aligned}
		\end{equation}
		For $I_{24}$ and $I_{25}$, recalling the definitions of $R_2$ and $R_4$ in \eqref{eq:sourse 3} and integrating by parts, we have
		\begin{equation}\label{eq:I24 25}
		\begin{aligned}
		I_{24}+I_{25}=&-\intx{\p_x^\a (R_2+R_4)\cdot\p_x^\a\psi}\\[5pt]
		=&-\intx{\p_x^\a \left( \phi \p_t u_0+(a+1)\psi\cdot\n_xu_0+\phi u_0\cdot\n_xu_0+u\cdot \n_x b\right) \cdot\p_x^\a\psi}\\[5pt]
		&-\intx{\p_x^\a \left( b\cdot\n_xu+\l v\cdot\n_x((\mathbf{I-P})g\M),v\r_v +\psi\cdot\nabla_x u_0\right)\cdot \p_x^\a \psi }\\[5pt]
        \leq& C\Hthreex{\phi}\Hthreex{\pt u_0}\Hthreex{\psi}+C(\Hthreex{a}+1)\Hthreex{\n_x u_0}\Hthreex{\psi}^2\\[5pt]
        &+C\Hfourx{u_0}^2\Hthreex{\phi}\Hthreex{\psi} +C\Hthreex{b}\Hthreex{\nabla_x \psi}(\Hthreex{\psi}+\Hthreex{u_0}) \\[5pt]
        &+ C (T)\HxxxLv{\mi g}\Hthreex{\n_x\psi}\\[5pt]
		\leq&\ \delta\Hxxx{\nabla_x\psi}^2+C_\delta(T)\E(t)+C_\delta(T)\Hxxxv{\mi g}^2,
		\end{aligned}
		\end{equation}
        where, in the first inequality above, we use
        \begin{equation*}\label{eq:mig-detail}
        \begin{aligned}
         &-\intx{\p_x^\a \l v\cdot\n_x((\mathbf{I-P})g\M),v\r_v \cdot \p_x^\a \psi }\\[5pt]
         =&\intx{\px^\a\left(\intv{v\otimes v\, \mi g \M}\right):\px^\a\n_x\psi}\\[5pt]
         =&\intx{\px^\a\left(\intv{(v-u)\otimes (v-u)\, \mi g \M}\right):\px^\a\n_x\psi}\\[5pt]
         &+\intx{\px^\a\left(\intv{(v-u)\otimes u\, \mi g \M}\right):\px^\a\n_x\psi}\\[5pt]
         &+\intx{\px^\a\left(\intv{u\otimes (v-u)\, \mi g \M}\right):\px^\a\n_x\psi}\\[5pt]
         &+\intx{\px^\a\left(\intv{u\otimes u\, \mi g \M}\right):\px^\a\n_x\psi}\\[5pt]
         \leq&C (\Hthreex{u}^5+1)\HxxxLv{\mi g}\Hthreex{\n_x\psi}\\[5pt]
         \leq& C(\Hthreex{\psi}^5+\Hxxx{u_0}^5+1)\HxxxLv{\mi g}\Hthreex{\n_x\psi}\\[5pt]
         \leq& C(T)\HxxxLv{\mi g}\Hthreex{\n_x\psi}.
        \end{aligned}
        \end{equation*}
		Collecting the estimates \eqref{eq:I21}, \eqref{eq:I22}, \eqref{eq:I23}, \eqref{eq:I24 25} together with \eqref{eq:psi bt}, and summing over $|\a| \leq 3$, we obtain \eqref{eq:psiH3}. This completes the proof.
	\end{proof}

	Next, we aim to derive the energy estimates for $b$.
	\begin{lemma}\label{lem:bH3} Under the conditions of Theorem \ref{thm1} and the \textit{a priori} assumption \eqref{assumption}, for $t\in [0,T_1] $, we have
		\begin{equation}\label{eq:bH3}
		\begin{aligned}
		&\sum\limits_{|\a|\leq 3}\left( \frac{1}{2}\dt{t}\intx{\frac{|\p_x^\a b|^2}{a+1}}+\intx{\frac{|\p_x^\a b|^2}{(a+1)\varepsilon}}+\intx{\frac{\p_x^\a\nabla_x\phi\cdot\p_x^\a b}{a+1}}+\intx{\p_x^\a\p_t\psi\cdot\p_x^\a b}\right) \\
		\leq&\ \delta\Hxxx{\nabla_x\psi}^2+C_\delta(T)\E(t)+C_\delta(T) \varepsilon\Hxxxxv{\mi g}^2+C_\delta (T)\varepsilon+\delta\frac{\Hxxx{b}^2}{\varepsilon},
		\end{aligned}
		\end{equation}
          where the energy functional $\E(t)$ is given in \eqref{defE} and $C_\delta(T) > 1$ is a constant only depending on $\delta,\ T$.
	\end{lemma}		
	\begin{proof}
		Applying $\p_x^\a\, (|\a|\leq 3)$ to $\eqref{eq:error}_2 $, multiplying by $\frac{\p_x^\a b}{a+1}$, integrating over $ \R^3$, we get
		\begin{equation}\label{eq:bH3-1}
		\begin{aligned}
		&\frac{1}{2}\dt{t}\intx{\frac{|\p_x^\a b|^2}{a+1}}+\intx{\frac{|\p_x^\a b|^2}{(a+1)\varepsilon}}+\intx{\frac{\p_x^\a\nabla_x\phi\cdot\p_x^\a b}{a+1}}+\intx{\p_x^\a\p_t\psi\cdot\p_x^\a b}\\
		=&\underbrace{-\intx{\frac{|\p_x^\a b|^2\p_t a}{2(a+1)^2}}}_{I_{31}}\underbrace{-\sum\limits_{\b\leq\a,|\b|\neq 0}C^\b_\a\intx{\frac{\p_x^\b a \p_x^{\a-\b}\p_t\psi\cdot\p_x^\a b}{a+1}}}_{I_{32}}\\
		&\underbrace{-\intx{\frac{\p_x^\a((a+1)u\cdot\n_x\psi)\cdot\p_x^\a b}{a+1}}}_{I_{33}}\underbrace{-\intx{\frac{ \p_x^\a (R_2+R_3) \cdot\p_x^\a b}{a+1}}}_{I_{34}}.\\
		\end{aligned}
		\end{equation}
		For $I_{31}$, using \eqref{eq:a_t Linfty}-\eqref{eq:a lower bound}, we have	
		\begin{equation}\label{eq:I31}
		I_{31}=-\intx{\frac{|\p_x^\a b|^2\p_t a}{2(a+1)^2}}\leq C\Linftyx{\pt a}\Hthreex{b}^2\leq C(T)\E(t).
		\end{equation}	
		For $I_{32}$, using $a=\rho_0+\phi$, we obtain
		\begin{equation}\label{eq:I32}
		\begin{aligned}
		I_{32}=&\sum\limits_{\b\leq\a,|\b|\neq 0}C^\b_\a\intx{\frac{\p_x^\b a \p_x^{\a-\b}\p_t\psi\cdot\p_x^\a b}{a+1}}\\
		\leq&\ C(\Hxxx{\phi}+\Hxxx{\rho_0})\Hxx{\p_t\psi}\Hxxx{b}\\
		\leq&\ \delta\frac{\Hxxx{b}^2}{\varepsilon}+C_\delta (T)\E(t).
		\end{aligned}
		\end{equation}
		For $I_{33}$, using \eqref{eq:a lower bound}, we have
		\begin{equation}\label{eq:I33}
		\begin{aligned}
		I_{33}=&-\intx{\frac{\p_x^\a((a+1)u\cdot\n_x\psi)\cdot\p_x^\a b}{a+1}}\\
		\leq&\ \delta\frac{\Hxxx{b}^2}{\varepsilon}+C_\delta\varepsilon(\Hthreex{a}^2+1)\Hthreex{u}^2\Hxxx{\nabla_x\psi}^2\\
		\leq&\ \delta\frac{\Hxxx{b}^2}{\varepsilon}+C_\delta (T)\E(t).
		\end{aligned}
		\end{equation}
		For $I_{34}$, recalling the definitions of $R_2$ and $R_3$ in \eqref{eq:sourse 3} and integrating by parts, we have
		\begin{equation}\label{eq:I34}
		\begin{aligned}
		I_{34}=&-\intx{\p_x^\a (R_2+R_3)\cdot\frac{\p_x^\a b}{a+1}}\\
		=&-\intx{\p_x^\a \left( \phi \p_t u_0+(a+1)\psi\cdot\n_xu_0+\phi u_0\cdot\n_xu_0+u\cdot \n_x b\right) \cdot \frac{\p_x^\a b}{a+1}}\\
		&-\intx{\p_x^\a \left( b\cdot\n_xu+\l v\cdot\n_x((\mathbf{I-P})g\M),v\r_v \right)\cdot \frac{\p_x^\a b}{a+1} }\\
		&-\intx{\p_x^\a \left((\rho_0+1)\p_t u_0+(\rho_0+1)u_0\cdot\n_xu_{0}+\n_x\rho_0 \right)\cdot \frac{\p_x^\a b}{a+1} }\\
		\leq&\ \delta\Hxxx{\nabla_x\psi}^2+\delta\frac{\Hthreex{b}^2}{\v}+C_\delta (T)\E(t)+C_\delta(T) \varepsilon\Hxxxxv{\mi g}^2+C_\delta(T)\varepsilon,
		\end{aligned}
		\end{equation}
        where in the inequality above,  we utilize
        \begin{equation*}
        \begin{aligned}
            &-\intx{(u\cdot \nabla_x \px^\a b)\cdot \frac{\px ^\a b}{a+1}}=\intx{\frac{1}{2}|\px^\a b|^2\div_x \left(\frac{u}{a+1}\right)}\leq C(T)\E(t),\\[5pt]
            &-\intx{(b\cdot \n_x\px^\a u)\cdot \frac{\px^\a b}{a+1}}=-\intx{(b\cdot \n_x\px^\a (\psi+u_0))\cdot \frac{\px^\a b}{a+1}}\leq \frac{\delta}{4}\Hthreex{\n_x\psi}^2+C_\delta (T)\E(t),
                \end{aligned}
        \end{equation*}
        and
        \begin{equation*}
            \begin{aligned}
                &-\intx{\p_x^\a \l v\cdot\n_x((\mathbf{I-P})g\M),v\r_v \cdot \frac{\p_x^\a b}{a+1}}\\[5pt]
                =&-\intxv{v\cdot \n_x \px^\a \left( \mi g\, \M\right)\frac{v\cdot \px^\a b}{a+1}}\\[5pt]
                =&-\intxv{(v-u)\cdot \n_x \px^\a \left( \mi g\, \M\right)\frac{(v-u)\cdot \px^\a b}{a+1}}\\[5pt]
                &-\intxv{u\cdot \n_x \px^\a \left( \mi g\, \M\right)\frac{(v-u)\cdot \px^\a b}{a+1}}\\[5pt]
                &-\intxv{(v-u)\cdot \n_x \px^\a \left( \mi g\, \M\right)\frac{u\cdot \px^\a b}{a+1}}\\[5pt]
                &-\intxv{u\cdot \n_x \px^\a \left( \mi g\, \M\right)\frac{u\cdot \px^\a b}{a+1}}\\[5pt]
                \leq&C(\Hthreex{u}^5+1)\left[\HxxLv{\mi g}\Hthreex{\nabla_x u}\Hthreex{b}+\HxxxxLv{\mi g}\Htwox{u}\Hthreex{b}\right]\\[5pt]
                \leq&\frac{\delta}{4}\Hthreex{\n_x\psi}^2+C_\delta(T)\E(t)+C_\delta(T)\v+\frac{\delta}{4} \frac{\Hthreex{b}^2}{\v}+C_\delta(T)\v \HxxxxLv{\mi g}^2.
            \end{aligned}
        \end{equation*}

		Finally, collecting the estimates \eqref{eq:bH3-1}, \eqref{eq:I31}, \eqref{eq:I32}, \eqref{eq:I33}, \eqref{eq:I34} and summing over $|\a| \leq 3$ gives the estimate \eqref{eq:bH3}.
	\end{proof}	

 To further eliminate the term \(\intx{ \partial_x^\alpha \partial_t\psi \cdot \partial_x^\alpha b}\) in \eqref{eq:bH3}, we prove the following lemma.
	\begin{lemma}\label{lem:psitH3}	Under the conditions of Theorem \ref{thm1} and the \textit{a priori} assumption \eqref{assumption}, for $t\in [0,T_1] $, we have
		\begin{equation}\label{eq:psitH3}
		\begin{aligned}
		&\sum\limits_{|\a|\leq 3}\left( \frac{\varepsilon}{2}\dt{t}\intx{|\nabla_x\p_x^\a\psi|^2}+\varepsilon\intx{|\p_x^\a\p_t\psi|^2}-\intx{ \p_x^\a\p_t\psi \cdot\p_x^\a b}\right) \\
		\leq&\ \delta\varepsilon\Hxxx{\p_t\psi}^2+C_\delta(T)\varepsilon+C_\delta(T)\E(t),
		\end{aligned}
		\end{equation}
          where the energy functional $\E(t)$ is given in \eqref{defE} and $C_\delta (T) > 1$ is a constant only depending on $\delta,\ T$.
	\end{lemma}
	\begin{proof}
	Applying $\partial_{x}^{\alpha}$ ($|\alpha|\leq 3$) to $\eqref{eq:error}_3$, multiplying by $\varepsilon\partial_{x}^{\alpha}\pt\psi$, and integrating by parts, we obtain
		\begin{equation*}
		\begin{aligned}
		&\frac{\varepsilon}{2}\dt{t}\intx{|\nabla_x\p_x^\a\psi|^2}+\varepsilon\intx{|\p_x^\a\p_t\psi|^2}-\intx{\p_x^\a b \cdot\p_x^\a\p_t\psi}\\[5pt]
		=&-\varepsilon\intx{\p_x^\a(u\cdot\nabla_x\psi)\cdot\p_x^\a\p_t\psi}+\varepsilon\intx{\p_x^\a(R_3-R_4)\cdot\p_x^\a\p_t\psi}\\[5pt]
		=&-\varepsilon\intx{\p_x^\a((\psi+u_0)\cdot\nabla_x\psi)\cdot\p_x^\a\p_t\psi}\\[5pt]
		&+\varepsilon\intx{\p_x^\a\big( (\rho_0+1)\p_t u_0+(\rho_0+1)u_0\cdot\n_xu_{0}+\n_x\rho_0-\psi\cdot\nabla_xu_0 \big)\cdot\p_x^\a\p_t\psi}\\[5pt]
		\leq&\ \delta\varepsilon\Hxxx{\p_t\psi}^2+C_\delta (T)\varepsilon+C_\delta (T)\E(t),
		\end{aligned}
		\end{equation*}
		where $\up
        {div}_x\psi=\div_x u-\div_x u_0=0$ is utilized in the first equality. The estimate \eqref{eq:psitH3} can be obtained by summing over $|\a| \leq 3$.
	\end{proof}

	To make it clear, we summarize the progress by combining  Lemmas \ref{lem:phiH3}-\ref{lem:psitH3}, yielding the partial energy estimate in the following corollary.
\begin{corollary}
\label{lem:phi b psi psit}
		Under the conditions of Theorem \ref{thm1} and the \textit{a priori} assumption \eqref{assumption}, for $t\in [0,T_1] $, we have
	\begin{equation}\label{eq:phi b psi psit}
	\begin{aligned}
	&\Hxxx{(\phi,\psi,b)}^2+\varepsilon\Hxxx{\nabla_x\psi}^2+\int_0^t\left( \Hxxx{\nabla_x\psi}^2+\frac{\Hxxx{b}^2}{\varepsilon}\right)\d s\\
	\leq&\ \int_0^t \Big( C_\delta (T)\E(s)+C_\delta(T)\varepsilon \Hxxxxv{\mi g}^2+C_\delta (T)\Hxxxv{\mi g}^2+C_\delta (T)\varepsilon \Big)\d s \\
	&+\int_0^t\Big( \delta\Hxx{\p_t\psi}^2+\delta\varepsilon\Hxxx{\p_t\psi}^2+ C_\delta(T)\varepsilon^2\Hxxx{\p_t\psi}^2 \Big) \d s+C\E(0)+C\varepsilon\Hxx{\pt\psi}^2,
	\end{aligned}
	\end{equation}
      where the energy functional $\E(t)$ is given in \eqref{defE} and $C_\delta(T) > 1$ is a constant only depending on $\delta,\ T$.
\end{corollary}
\begin{proof}
	First, combining \eqref{eq:phiH3} in Lemma \ref{lem:phiH3} and \eqref{eq:bH3} in Lemma \ref{lem:bH3} yields
	\begin{equation}\label{eq:phi b H3}
	\begin{aligned}
	&\sum\limits_{|\a|\leq 3}\left( \frac{1}{2}\dt{t}\intx{\frac{|\p_x^\a \phi|^2}{a+1}} + \frac{1}{2}\dt{t}\intx{\frac{|\p_x^\a b|^2}{a+1}}+\intx{\frac{|\p_x^\a b|^2}{(a+1)\varepsilon}}+\intx{\p_x^\a\p_t\psi\cdot\p_x^\a b}\right)\\
	\leq&\ \delta\Hxxx{\nabla_x\psi}^2+C_\delta (T)\E(t)+C_\delta (T)\varepsilon \Hxxxxv{\mi g}^2+C_\delta (T)\varepsilon+\delta\frac{\Hxxx{b}^2}{\varepsilon},
	\end{aligned}
	\end{equation}
	where we have used the fact that
	\begin{equation*}
	\begin{aligned}
	&\left|\intx{\frac{\p_x^\a\div_xb\,\p_x^\a\phi}{a+1}}+\intx{\frac{\p_x^\a\nabla_x\phi\cdot\p_x^\a b}{a+1}}\right|\\[5pt]
    =&\left|\intx{\frac{\div_x(\p_x^\a b\,\p_x^\a\phi)}{a+1}}\right|
	=\left|\intx{\frac{(\p_x^\a b\,\p_x^\a\phi)\cdot\nabla_xa}{(a+1)^2}}\right|
    \leq C(T)\E(t).
	\end{aligned}
	\end{equation*}
	
	Then, summing \eqref{eq:psiH3} in Lemma \ref{lem:psiH3} and \eqref{eq:psitH3} in Lemma \ref{lem:psitH3} gives
	\begin{equation}\label{eq:psi psit H3}
	\begin{aligned}
	&\sum\limits_{|\a|\leq 3}\left( \frac{1}{2}\dt{t}\intx{(a+2)|\p_x^\a\psi|^2}+\intx{|\p_x^\a\n_x\psi|^2} \right) \\[5pt]
	&+\sum\limits_{|\a|\leq 3}\left( \varepsilon\dt{t}\intx{\p_x^\a\p_t\psi\cdot\p_x^\a\psi}+\varepsilon\dt{t}\intx{|\p_x^\a\nabla_x\psi|^2}-\intx{\p_x^\a\p_t\psi \cdot\p_x^\a b}\right) \\[5pt]
	\leq&\ \delta\Hxxx{\nabla_x\psi}^2+\delta\Hxx{\p_t\psi}^2+C_\delta (T)\E(t)+C_\delta (T)\Hxxxv{\mi g}^2\\[5pt]
	&+C_\delta(T)\varepsilon^2\Hxxx{\p_t\psi}^2+C_\delta (T)\varepsilon+ \delta\varepsilon\Hxxx{\p_t\psi}^2.
	\end{aligned}
	\end{equation}

	Furthermore, by adding \eqref{eq:phi b H3} and \eqref{eq:psi psit H3}, we obtain
	\begin{equation}\label{eq:totalma}
	\begin{aligned}
	&\sum\limits_{|\a|\leq 3}\left( \frac{1}{2}\dt{t}\intx{\frac{|\p_x^\a \phi|^2}{a+1}} + \frac{1}{2}\dt{t}\intx{\frac{|\p_x^\a b|^2}{a+1}}+\intx{\frac{|\p_x^\a b|^2}{(a+1)\varepsilon}}\right)\\[5pt]
	&+\sum\limits_{|\a|\leq 3}\left( \frac{1}{2}\dt{t}\intx{(a+2)|\p_x^\a\psi|^2}+\intx{|\p_x^\a\n_x\psi|^2} \right) \\[5pt]
	&+\sum\limits_{|\a|\leq 3}\left( \varepsilon\dt{t}\intx{\p_x^\a\p_t\psi\cdot\p_x^\a\psi}+\varepsilon\dt{t}\intx{|\p_x^\a\nabla_x\psi|^2}\right)\\[5pt]
	\leq&\ \delta\Hxxx{\nabla_x\psi}^2+C_\delta (T)\E(t)+C_\delta(T) \varepsilon \Hxxxxv{\mi g}^2+C_\delta (T)\Hxxxv{\mi g}^2\\[5pt]
    &+C_\delta (T)\varepsilon+\delta\frac{\Hxxx{b}^2}{\varepsilon}+\delta\Hxx{\p_t\psi}^2+C_\delta(T)\varepsilon^2\Hxxx{\p_t\psi}^2+ \delta\varepsilon\Hxxx{\p_t\psi}^2.
	\end{aligned}
	\end{equation}

    Finally, considering
    \begin{equation*}
	\begin{aligned}
	&\sum\limits_{|\a|\leq 3}\left|\varepsilon\intx{\p_x^\a\p_t\psi\cdot\p_x^\a\psi}\right|\leq C\varepsilon\Hxx{\pt\psi}^2+C_\delta\varepsilon\Hxx{\psi}^2+\delta\varepsilon\Hxxx{\nabla_x\psi}^2,
	\end{aligned}
	\end{equation*}
    choosing $\delta$ sufficiently small and integrating \eqref{eq:totalma} over $(0,t)$, we can obtain the estimate \eqref{eq:phi b psi psit}.
\end{proof}
	
In order to control the problematic term $\p_t\psi$ on the right-hand side of Corollary \ref{lem:phi b psi psit}, we have to further make use of the dissipation for $\p_t\psi$, which is presented in the lemma below.

	\begin{lemma}\label{lem:bt psit}
			Under the conditions of Theorem \ref{thm1} and the \textit{a priori} assumption \eqref{assumption}, for $t\in [0,T_1] $, we have
		\begin{equation*}
		\begin{aligned}
		&\frac{\Hxx{b}^2}{\varepsilon}+\Hxx{\nabla_x\psi}^2+\varepsilon\Hxx{\p_t\psi}^2+\int_0^t{\left( \Hxx{\p_tb}^2+\Hxx{\p_t\psi}^2+\varepsilon\Hxx{\nabla_x\p_t\psi}^2\right) }\d s\\
		\leq&\ C(T)\E(0)+C(T)\varepsilon+C(T)\int_{0}^t\left( \E(s)+\Hxxxv{\mi g}^2 \right)\d s,
		\end{aligned}
		\end{equation*}
          where the energy functional $\E(t)$ is given in \eqref{defE} and $C(T) > 1$ is a constant only depending on $T$.
	\end{lemma}

	\begin{proof}
    \textbf{Step 1: Zero-th order derivative estimate.} Multiplying $\eqref{eq:error}_2 $ and $\eqref{eq:error}_4$  by $\frac{\pt b}{a+1}$ and $ \pt \psi$, respectively, integrating over $ \R^3$, we obtain
		\begin{equation}\label{eq:bL2}
		\begin{aligned}
		&\frac{1}{2}\dt{t}\intx{\frac{|b|^2}{\varepsilon(a+1)}}+\intx{\frac{|\pt b|^2}{a+1}}+\frac{1}{2}\dt{t}\Lx{\nabla_x\psi}^2+\intx{(a+2)|\p_t\psi|^2}+\underbrace{\intx{2\p_t\psi\cdot \p_tb}}_{I_{40}}\\
		=&\underbrace{-\intx{\frac{|b|^2\pt a}{2\varepsilon(a+1)^2}}}_{I_{41}}\underbrace{-\intx{[(a+1)u\cdot\nabla_x\psi+\nabla_x\phi]\cdot \frac{\pt b}{a+1}}}_{I_{42}}\underbrace{-\intx{(R_2+R_3)\cdot \frac{\pt b}{a+1}}}_{I_{43}}\\
		&\underbrace{-\intx{[(a+2)u\cdot\nabla_x\psi+\nabla_x\phi]\cdot \p_t\psi}}_{I_{44}}\underbrace{-\intx{(R_2+R_4)\cdot \p_t\psi}}_{I_{45}}.
		\end{aligned}
		\end{equation}
		
		We first estimate $I_{40}$. Using $\eqref{eq:error}_3$, we have
		\begin{equation}\label{eq:111}
		\p_t b=\varepsilon \p_{tt}\psi-\varepsilon\Delta_x\p_t\psi+\varepsilon\p_t(u\cdot\nabla_x\psi)+\varepsilon\p_t\nabla_x(P+\rho_0-P_0)-\varepsilon\p_tR_3+\varepsilon\p_tR_4.
		\end{equation}
		Substituting \eqref{eq:111} into $\intx{2\p_t\psi\cdot \p_tb}$ yields
		\begin{equation}\label{eq:btpsit-0}
		\begin{aligned}
		&\intx{2\p_t\psi\cdot \p_tb}\\
		=&\intx{ 2\p_t\psi\cdot \left(\varepsilon \p_{tt}\psi-\varepsilon\Delta_x\p_t\psi+\varepsilon\p_t(u\cdot\nabla_x\psi)+\varepsilon\p_t\nabla_x(P+\rho_0-P_0)-\varepsilon\p_tR_3+\varepsilon\p_tR_4 \right) }\\
		=&\ \varepsilon\dt{t}\intx{|\p_t\psi|^2}+2\varepsilon \intx{|\nabla_x \p_t\psi|^2}+\underbrace{\varepsilon\intx{ 2\p_t\psi\cdot \p_t(u\cdot\nabla_x\psi) }}_{J_{21}}\\
		&+\underbrace{\varepsilon\intx{2 \p_t\psi\cdot \p_t(-R_3+R_4) }}_{J_{22}}.
		\end{aligned}
		\end{equation}
		For $J_{21}$, by integrating by parts, we have
		\begin{equation}\label{eq:J21}
		\begin{aligned}
		\left|J_{21}\right|=&\ \varepsilon\left|\intx{ \p_t\psi\cdot \p_t(u\cdot\nabla_x\psi) }\right|\\[5pt]
		\leq & \ C\varepsilon\left|\intx{ \p_t\psi\cdot (\p_t (\psi+u_0)\cdot\nabla_x\psi) }\right|\\[5pt]
		\leq&\ C\varepsilon\Hxx{\p_t\psi}\Big(\Hxx{\p_t\psi}+\Hxx{\p_t u_0}\Big)\Hxx{\nabla_x\psi}
		\leq C(T)\E(t).
		\end{aligned}
		\end{equation}
		For $J_{22}$, we have
		\begin{equation}\label{eq:J22}
		\begin{aligned}
		\left|J_{22}\right|=&\left| \varepsilon\intx{ \p_t\psi\cdot \p_t(R_3 +R_4)} \right|\\[5pt]
		=&\left| \varepsilon\intx{\p_t\psi\cdot \p_t\big( (\rho_0+1)\p_t u_0+(\rho_0+1)u_0\cdot\n_xu_{0}+\n_x\rho_0-\psi\cdot\nabla_xu_0\big) } \right|\\[5pt]
		\leq&\ C(T)\varepsilon\Hxx{\p_t\psi}^2+C(T)\varepsilon\\[5pt]
		\leq&\ C(T)\E(t)+C(T)\varepsilon.
		\end{aligned}
		\end{equation}
		Combining \eqref{eq:btpsit-0}-\eqref{eq:J22}, we find
		\begin{equation}\label{eq:bH3-left}
		\begin{aligned}
		\intx{2\p_t\psi\cdot \p_tb}\geq&\ \varepsilon\dt{t}\intx{|\p_t\psi|^2}+2\varepsilon \intx{|\nabla_x \p_t\psi|^2}-\left(C(T)\E(t) +C(T)\varepsilon\right).
		\end{aligned}
		\end{equation}
		Next, we proceed to estimate $I_{41}$ to $I_{45}$. For $I_{41},\ I_{42}$ and $I_{44}$, we have
		\begin{equation}\label{eq:I414244}
		\begin{aligned}
		|I_{41}|+|I_{42}|+|I_{44}|	\leq&\ \delta \Big(\Lx{\p_tb}^2+\Lx{\p_t\psi}^2\Big) + C_\delta (T)\E(t).
		\end{aligned}
		\end{equation}
		For $I_{43}$ and $I_{45}$, we obtain
		\begin{equation}\label{eq:I4345-1}
		\begin{aligned}
		&I_{43}+I_{45}\\[5pt]
        =&-\intx{(R_2+R_3)\cdot \frac{\pt b}{a+1}}-\intx{(R_2+R_4)\cdot \p_t\psi}\\[5pt]
		\leq&\left|\intx{\left(\phi \p_t u_0+(a+1)\psi\cdot\n_xu_0+\phi u_0\cdot\n_xu_0+u\cdot \n_x b \right) \cdot \left(\frac{ \pt b}{a+1}+\p_t\psi\right)} \right|\\[5pt]
		&+\left|\intx{\left( b\cdot\n_xu+\l v\cdot\n_x((\mathbf{I-P})g\M),v-u\r_v\right) \cdot \left(\frac{\pt b}{a+1}+\p_t\psi\right)} \right|\\[5pt]
		&+\left|\intx{\psi\cdot\nabla_xu_0\cdot\p_t\psi} \right|\underbrace{-\intx{\left((\rho_0+1)\p_t u_0+(\rho_0+1)u_0\cdot\n_xu_{0}+\n_x\rho_0 \right) \cdot \frac{ \pt b}{a+1}} }_{K_{1}}\\[5pt]
		\leq&\ \delta\Big(\Lx{\p_tb}^2+\Lx{\p_t\psi}^2\Big) + C_\delta (T)\E(t)+C_\delta (T)\Hxv{\mi g}^2+K_1.
		\end{aligned}
		\end{equation}
	For $K_{1} $, we have
	\begin{equation}\label{eq:I4345-2}
	\begin{aligned}
	K_{1}=&-\intx{\left((\rho_0+1)\p_t u_0+(\rho_0+1)u_0\cdot\n_xu_{0}+\n_x\rho_0 \right) \cdot\frac{\pt b}{a+1}}\\[5pt]
	=&-\dt{t}\intx{\left((\rho_0+1)\p_t u_0+(\rho_0+1)u_0\cdot\n_xu_{0}+\n_x\rho_0 \right) \cdot \frac{b}{a+1}}\\[5pt]
	&+\intx{\pt\left[\left((\rho_0+1)\p_t u_0+(\rho_0+1)u_0\cdot\n_xu_{0}+\n_x\rho_0 \right)\frac{1}{a+1} \right]\cdot b}\\[5pt]
	\leq&-\dt{t}\intx{R_3\cdot \frac{b}{a+1}}+C(T)\varepsilon+C\E(t).
	\end{aligned}
	\end{equation}	
	Hence, combining \eqref{eq:bL2}, \eqref{eq:bH3-left}, \eqref{eq:I414244}, \eqref{eq:I4345-1} and \eqref{eq:I4345-2}, we have
		\begin{equation}\label{eq:btpsit}
		\begin{aligned}
			&\frac{1}{2}\dt{t}\intx{\frac{|b|^2}{\varepsilon(a+1)}}+\intx{\frac{|\pt b|^2}{a+1}}+\frac{1}{2}\dt{t}\Lx{\nabla_x\psi}^2+\intx{(a+2)|\p_t\psi|^2}\\[5pt]
		&\qquad \qquad \qquad +\varepsilon\dt{t}\intx{|\p_t\psi|^2}+2\varepsilon \intx{|\nabla_x \p_t\psi|^2}+\dt{t}\intx{R_3 \cdot \frac{b}{a+1}}\\[5pt]
		\leq&\ C_\delta (T)\E(t)+C_\delta (T)\varepsilon+\delta(\Lx{\p_tb}^2+\Lx{\p_t\psi}^2)+C_\delta (T)\Hxv{\mi g}^2.
		\end{aligned}
		\end{equation}

		By further choosing $\delta$ sufficiently small, integrating \eqref{eq:btpsit} over $(0,t)$ and considering \eqref{eq:a lower bound}, we can obtain
		\begin{equation*}
		\begin{aligned}
		&\frac{\Lx{b}^2}{\varepsilon}+\Lx{\nabla_x\psi}^2+\varepsilon\Lx{\p_t\psi}^2+\int_0^t{\left( \Lx{\p_tb}^2+\Lx{\p_t\psi}^2+\varepsilon\Lx{\nabla_x\p_t\psi}^2\right) }\d s\\
		\leq&\ C\E(0)+C(T)\varepsilon+ C(T)\int_{0}^t\left(\E(s)+\Hxv{\mi g}^2 \right)\d s,
		\end{aligned}
		\end{equation*}
		where we also use
		\begin{equation*}
		\left|\intx{R_3 \cdot \frac{b}{a+1}}\right|\leq \delta\frac{\Lx{b}^2}{\varepsilon}+C_\delta (T)\varepsilon.
		\end{equation*}

        \textbf{Step 2: Higher-order derivative estimates.}
		To estimate the first-order derivative terms, we apply $\p_x^\a$ $(|\a|=1)$ to $\eqref{eq:error}_2 $ and $\eqref{eq:error}_4$, multiply by $\frac{\px^\a \pt b}{a+1}$ and $\px^\a \pt \psi$, respectively, and then integrate them over $ \R^3$,
		\begin{equation}\label{I51only}
		\begin{aligned}
		&\frac{1}{2}\dt{t}\intx{\frac{|\px^\a b|^2}{\varepsilon(a+1)}}+\intx{\frac{|\px^\a \pt b|^2}{a+1}}+\frac{1}{2}\dt{t}\Lx{\p_x^\a\nabla_x\psi}^2\\[5pt]
		&\qquad \qquad \qquad \qquad \qquad +\intx{(a+2)|\p_x^\a\p_t\psi|^2}+\intx{2\p_x^\a\p_t\psi\cdot \p_x^\a\p_tb}\\[5pt]
		=&\underbrace{-\intx{\p_x^\a a\, \p_t\psi\cdot\frac{\px^\a \pt b}{a+1}}-\intx{\p_x^\a a\, \p_t\psi\cdot\p_x^\a\p_t\psi}}_{I_{51}}-\intx{\frac{|\px^\a b|^2 \pt a}{\varepsilon(a+1)^2}}\\[5pt]
		&-\intx{\p_x^\a[(a+1)u\cdot\nabla_x\psi+\nabla_x\phi]\cdot \frac{\px^\a \pt b}{a+1}}-\intx{\p_x^\a(R_2+R_3)\cdot \frac{\px^\a \pt b}{a+1}}\\[5pt]
		&-\intx{\p_x^\a[(a+2)u\cdot\nabla_x\psi+\nabla_x\phi]\cdot \p_x^\a\p_t\psi}-\intx{\p_x^\a(R_2+R_4)\cdot\p_x^\a \p_t\psi}.
		\end{aligned}
		\end{equation}
		Here, we only detail the estimate for $I_{51}$:
		\begin{equation*}
		\begin{aligned}
		I_{51}=&-\intx{\p_x^\a a\, \p_t\psi\cdot\frac{\p_x^\a\, \p_tb}{a+1}}-\intx{\p_x^\a a\, \p_t\psi\cdot\p_x^\a\p_t\psi}\\[5pt]
		\leq&\ \delta \Big(\Lx{\p_x^\a \p_tb}^2+\Lx{\p_x^\a\p_t \psi}^2\Big) + C_\delta (T)\Lx{\p_t\psi}^2.
		\end{aligned}
		\end{equation*}
		and by analogous estimates for the remaining terms on the right side of \eqref{I51only}, we have
		\begin{equation}\label{eq:btpsit-H1}
		\begin{aligned}
   	&\frac{1}{2}\dt{t}\intx{\frac{|\px^\a b|^2}{\varepsilon(a+1)}}+\intx{\frac{|\px^\a \pt b|^2}{a+1}}+\frac{1}{2}\dt{t}\Lx{\p_x^\a\nabla_x\psi}^2\\[5pt]
		&+\intx{(a+2)|\p_x^\a\p_t\psi|^2}+\varepsilon\dt{t}\intx{| \px^\a \p_t\psi|^2}+2\varepsilon \intx{|\nabla_x \px^\a\p_t\psi|^2}\\[5pt]
		\leq&\ C_\delta (T)\E(t)+\delta\varepsilon\Lx{\nabla^2_x\p_t\psi}^2 +C_\delta(T)\varepsilon + \delta\Big(\Lx{\p_x^\a\p_tb}^2+\Lx{\p_x^\a\p_t\psi}^2\Big)\\[5pt]
		&+C_\delta(T)\Hxxv{\mi g}^2+C_\delta(T)\Lx{\p_t\psi}^2,
		\end{aligned}
		\end{equation}
      where, note that the energy and dissipation terms
      \[\varepsilon\dt{t}\intx{\vert{} \px^\a \p_t\psi\vert{}^2}+2\varepsilon \intx{\vert{}\nabla_x \px^\a\p_t\psi\vert{}^2}\]
      arise from $\intx{2\p_x^\a\p_t\psi\cdot \p_x^\a\p_tb}$ by a similar argument as in \eqref{eq:btpsit-0}.

		By further integrating \eqref{eq:btpsit-H1} over $(0,t)$ and considering \eqref{eq:a lower bound}, we obtain
		\begin{equation*}\label{H1-x-b}
		\begin{aligned}
		&\frac{\Lx{\nabla_xb}^2}{\varepsilon}+\Lx{\nabla^2_x\psi}^2+\varepsilon\Lx{\nabla_x\p_t\psi}^2+\int_0^t \left( \Lx{\nabla_x\p_tb}^2+\Lx{\nabla_x\p_t\psi}^2+\varepsilon\Lx{\nabla^2_x\p_t\psi}^2\right) \d s\\
		\leq&\ C\E(0)+C(T)\varepsilon+C(T)\int_{0}^t\left( \E(s)+\Hxxv{\mi g}^2 \right)\d s.
		\end{aligned}
		\end{equation*}
        In fact, the second-order estimate can be completed by applying $\p_x^\a$ $(|\a|=2)$ to $\eqref{eq:error}_2 $ and $\eqref{eq:error}_4$ with completely analogous arguments as in the first-order estimate. Thus, we complete $H^2_x$ estimates by collecting $0 \leq |\alpha| \leq 2$, and thereby prove the lemma.
	\end{proof}

	The following lemma aims to establish the dissipation estimate for $\Lx{\nabla_x^5\psi}$, which will play a crucial role in the subsequent estimate of $g$
	\begin{lemma}\label{lem:psiH4}		
    Under the conditions of Theorem \ref{thm1} and the \textit{a priori} assumption \eqref{assumption}, for $t\in [0,T_1] $, one has
		\begin{equation*}
		\varepsilon\Lx{\nabla_x^4\psi}^2+\varepsilon\intt{\Lx{\nabla_x^5\psi}^2}{s}\leq C\E(0)+C(T)\varepsilon+C(T)\intt{\left(\E(s)+\frac{\Lx{\nabla_x^3b}^2}{\varepsilon} \right) }{s},
		\end{equation*}
          where the energy functional $\E(t)$ is given in \eqref{defE} and $C(T) > 1$ is a constant only depending on $T$.
	\end{lemma}
	\begin{proof}
		This lemma can be proved by directly applying $\p^\a_x$ with $ |\a|=4 $ to  $\eqref{eq:error}_3$, multiplying by $\v\p^\a_x\psi$, and integrating over $x, t$.
	\end{proof}

The following corollary is a summary of Corollary \ref{lem:phi b psi psit} and Lemmas \ref{lem:bt psit}, \ref{lem:psiH4}.

\begin{corollary}
\label{coro:ma total}	
Under the conditions of Theorem \ref{thm1} and the \textit{a priori} assumption \eqref{assumption}, for $t\in [0,T_1] $, we have
	\begin{equation}\label{eq:ma total}
	\begin{aligned}
	&\Hxxx{(\phi,\psi,b)}^2+\frac{\Hxx{b}^2}{\varepsilon}+\varepsilon\Hxx{\p_t\psi}^2+\varepsilon\Hxxx{\nabla_x\psi}^2\\
	&+\int_0^t\left( \Hxxx{\nabla_x\psi}^2+\varepsilon\Hxxxx{\nabla_x\psi}^2+\frac{\Hxxx{b}^2}{\varepsilon}+\Hxx{\p_tb}^2+\Hxx{\p_t\psi}^2+\varepsilon\Hxx{\nabla_x\p_t\psi}^2\right)\d s\\
	\leq& C(T)\int_0^t \left(  \E(s)+\varepsilon\Hxxxxv{\mi g}^2+\Hxxxv{\mi g}^2\right)\d s+C(T)\E(0)+C(T)\varepsilon,
	\end{aligned}
	\end{equation}
where the energy functional $\E(t)$ is given in \eqref{defE} and $C(T) > 1$ is a constant only depending on $T$.
\end{corollary}	
\begin{proof}
	The corollary is a direct consequence of  Corollary \ref{lem:phi b psi psit} and Lemmas \ref{lem:bt psit}, \ref{lem:psiH4} by selecting sufficiently small $\delta$ and $\varepsilon$.
\end{proof}

%---------------------

By observing the estimate \eqref{eq:ma total}, we find it remains to estimate $\varepsilon\Hxxxxv{\mi g}^2$ and $\Hxxxv{\mi g}^2$ in order to close the total energy estimate.

Now, we first derive the dissipation estimate for $\varepsilon\Hxxxxv{\mi g}^2$.
To simplify the analysis, rather than deriving estimates directly for $g$, we introduce
\[
f=M_u+g\sqrt{M_u}=h\sqrt{M_u} \quad \text{and}\quad \rho=\int h \sqrt{M_u} \d v,
\]
in fact, a direct calculation reveals that
\begin{equation}\label{eq:relation-g-h}
\rho=a+1, \quad (\mathbf{I-P})h=(\mathbf{I-P})g, \quad h=\rho\sqrt{M_u}+b\cdot(v-u)\sqrt{M_u}+(\mathbf{I-P})h,
\end{equation}
and
\begin{equation*}
\int  h(v-u)\M \d v=\int (g+\M)(v-u)\M \d v = b,
\end{equation*}
such that $h$ are derived to satisfy:
\begin{equation}
\left\{
\begin{array}{ll}
\label{h u}
\pt h+v\cdot \nabla_xh+\frac{1}{\varepsilon}\mathbf{L}h+\frac{1}{2}h(v-u)\cdot \pt u+\frac{1}{2}h(v\cdot\nabla_xu)\cdot(v-u)=0,\\[2mm]
\pt u+ u\cdot\nabla_xu+\nabla_xP -\Delta_x u -\frac{b}{\varepsilon}=0.\\
\end{array}
\right.
\end{equation}

Throughout the subsequent analysis, we denote,
 for $|\a|\geq 1$,
\begin{equation*}
\rho_\a:=\int \px^\a h\M \d v,\ \quad  b_\a:=\int \px^\a h\, (v-u)\M \d v,
\end{equation*}
and we will repeatedly use the relations $\rho=a+1$ and $u=\psi+u_0$, alongside the operator
\begin{equation*}
\mathbf{L}h=-\Delta_vh+\frac{|v-u|^2}{4}h-\frac{3}{2}h,
\end{equation*}
and furthermore, for $1\leq |\alpha|\leq 4$,
\begin{equation}\label{Lx gamma}
\partial^\alpha_x\mathbf{L}h=\mathbf{L}(\partial^\alpha_xh) + \sum\limits_{\beta\leq\alpha,|\b|\neq 0} C^{\beta}_{\alpha} \partial^\beta_x\left(\frac{|v-u|^2}{4}\right)\partial^{\alpha-\beta}_xh.
\end{equation}

For brevity, summation symbols are omitted in the subsequent equations. The indices $\alpha$ and $\beta$ implicitly denote summations over $1 \leq |\alpha| \leq 4$ and $\beta\leq \alpha$ (with $|\beta|\neq 0$), respectively.

\begin{lemma}\label{lem:hH4}
Under the conditions of Theorem \ref{thm1} and the \textit{a priori} assumption \eqref{assumption}, for $t\in [0,T_1] $, we have
\begin{equation}\label{eq:hH4}
\begin{aligned}
&\varepsilon^2\HxxxLv{\nabla_xh}^2+\varepsilon^3\Hthreex{\nabla_x^2 u}^2\\
&+\intt{\left(\varepsilon^3\Hthreex{\nabla_x\pt u}^2+\sum\limits_{1\leq|\alpha|\leq4}\varepsilon(\Ltwox{b_\alpha}^2+\Lv{\mi \px^\alpha h}^2)\right)}{s}\\
\leq&C\E(0)+C_\delta(T)\varepsilon+C_\delta (T)\intt{\E(s)}{s}+C_\delta (T)\intt{\sum\limits_{|\theta|\leq3}\Lv{\px^\theta\mi h}^2}{s}\\
&+C\intt{(\delta\varepsilon\Ltwox{\nabla_x^5\psi}^2+\delta\varepsilon\Hthreex{\pt \psi}^2)}{s},
\end{aligned}
\end{equation}
where the energy functional $\E(t)$ is given in \eqref{defE} and $C_\delta(T) > 1$ is a constant only depending on $\delta,\ T$.
\end{lemma}

\begin{proof}
	Applying $\partial_{x}^{\alpha}\ (1\leq|\alpha|\leq4)$ to $\eqref{h u}_1$ and $\eqref{h u}_2$, multiplying the results by $\varepsilon^2\partial_{x}^{\alpha}h$ and by $\varepsilon^3\partial_{x}^{\alpha}\pt u$ respectively, then adding them together and integrating by parts, we obtain
	\begin{equation*}
	\begin{aligned}
	&\frac{1}{2}\dt{t}\left(\varepsilon^2\Ltwoxv{\px^\alpha h}^2+\varepsilon^3\Ltwox{\px^\alpha\nabla_xu}^2\right) +\varepsilon^3\Ltwox{\px^\alpha \pt u}^2+\varepsilon\Ltwox{b_\alpha}^2+\varepsilon c\Lv{(\mathbf{I-P})\partial_{x}^\alpha h}^2\\[5pt]
	\leq& \underbrace{-\varepsilon^2\frac{1}{2}\inxv{\px^\alpha[h(v-u)\cdot \pt u],\px^\alpha h}+\varepsilon^2\inx{\px^\alpha b, \px^\alpha \pt u}}_{J_{31}}\underbrace{ -\varepsilon^2\frac{1}{2}\inxv{\px^\alpha[h (v\cdot\nabla_x u)\cdot(v-u)],\px^\alpha h}}_{J_{32}}\\[5pt]
	&\underbrace{-\varepsilon^3\inx{\partial^\alpha_x(u\cdot \nabla_xu),\partial^\alpha_x\pt u}}_{J_{33}} \underbrace{-\varepsilon C^\beta_\alpha\inxv{\partial^\beta_x\left(\frac{|v-u|^2}{4}\right)\partial^{\alpha-\beta}_xh,\partial^\alpha_xh}}_{J_{34}},
	\end{aligned}
	\end{equation*}	
    where we have used \eqref{Lx gamma} and Lemma \ref{lemma A2}.

	For $J_{31}$,  using $\px^\alpha h=\rho_\alpha\sqrt{M_u}+b_\alpha\cdot(v-u)\sqrt{M_u}+(\mathbf{I-P})\px^\alpha h $, we obtain
	\begin{equation*}
	\begin{aligned}
	J_{31}=&	-\varepsilon^2\frac{1}{2}\inxv{\partial^\alpha_x(h(v-u)\cdot \pt u),\partial^\alpha_xh}+\varepsilon^2\inx{\px^\alpha b, \px^\alpha \pt u}\\[5pt]
	=&-\frac{1}{2} \varepsilon^2\inxv{h(v-u)\cdot \partial^\alpha_x\pt u,\partial^\alpha_xh} - \frac{1}{2}C^\beta_\alpha\varepsilon^2\inxv{\partial^\beta_x(h(v-u))\cdot \partial^{\alpha-\beta}_x\pt u,\partial^\alpha_xh} +\varepsilon^2\inx{\px^\alpha b, \px^\alpha \pt u}\\[5pt]
	=&\left( -\frac{1}{2} \varepsilon^2\inxv{h(v-u)\cdot \partial^\alpha_x\pt u,\rho_\alpha\sqrt{M_u}+b_\alpha\cdot(v-u)\sqrt{M_u}+(\mathbf{I-P})\partial^\alpha_xh} +\varepsilon^2\inx{\px^\alpha b, \px^\alpha \pt u}\right) \\[5pt]
	&-\frac{1}{2}C^\beta_\alpha \varepsilon^2\inxv{\partial^\beta_x(h(v-u))\cdot \partial^{\alpha-\beta}_x\pt u,\rho_\alpha\sqrt{M_u}+b_\alpha\cdot(v-u)\sqrt{M_u}+(\mathbf{I-P})\partial^\alpha_xh} \\[5pt]
	=&J_{311}+J_{312}.
	\end{aligned}
	\end{equation*}
	For $J_{311}$, using \eqref{assumption} and \eqref{eq:relation-g-h}, we have
	\begin{equation}\label{eq:J311-1}
	\begin{aligned}
	J_{311}=&-\frac{1}{2} \varepsilon^2\inxv{h(v-u)\cdot \partial^\alpha_x\pt u,\rho_\alpha\sqrt{M_u}+b_\alpha\cdot(v-u)\sqrt{M_u}+(\mathbf{I-P})\partial^\alpha_xh} \\[5pt]
    &+\varepsilon^2\inx{\px^\alpha b, \px^\alpha \pt u} \\[5pt]
	=&-\frac{1}{2} \varepsilon^2\inxv{b\cdot(v-u)\sqrt{M_u}(v-u)\cdot \partial^\alpha_x\pt u,\rho_\alpha\sqrt{M_u}}+\varepsilon^2\inx{\px^\alpha b, \px^\alpha \pt u} \\[5pt]
	&-\frac{1}{2} \varepsilon^2\inxv{(\rho\sqrt{M_u}+(\mathbf{I-P})h)(v-u)\cdot \partial^\alpha_x\pt u,b_\alpha\cdot(v-u)\sqrt{M_u}} \\[5pt]
	&-\frac{1}{2} \varepsilon^2\inxv{(b\cdot(v-u)\sqrt{M_u}+(\mathbf{I-P})h)(v-u)\cdot \partial^\alpha_x\pt u,(\mathbf{I-P})\partial^\alpha_xh} \\[5pt]
	\leq&-\varepsilon^2\inx{\left(\frac{\rho}{2}-1\right)\px^\alpha \pt u,\px^\alpha b}+C\varepsilon^2\left|\inx{\px^\alpha \pt u,b_\alpha-\px^\alpha b} \right|+C_\delta(T) \E(t)\\[5pt]
	&+C (\Hxx{b}+\Hxxv{\mi g})(\varepsilon^3\Ltwox{\px^\alpha \pt u}^2+\varepsilon\Ltwox{b_\alpha}^2+\varepsilon\Lv{(\mathbf{I-P})\partial_{x}^\alpha h}^2)\\[5pt]
	&+\delta(\varepsilon^3\Ltwox{\px^\alpha \pt u}^2+\varepsilon\Ltwox{b_\alpha}^2)\\[5pt]
	\leq&-\varepsilon^2\inx{\left(\frac{\rho}{2}-1\right)\px^\alpha \pt u,\px^\alpha b} + C_\delta (T)\varepsilon+C_\delta (T)\E(t) \\[5pt]
	&+C(\delta+\eta)(\varepsilon^3\Ltwox{\px^\alpha \pt u}^2+\varepsilon\Ltwox{b_\alpha}^2+\varepsilon\Lv{(\mathbf{I-P})\partial_{x}^\alpha h}^2),
	\end{aligned}
	\end{equation}		
	where we have used (see Lemma \ref{lem:b-b_a})
	\begin{equation*}
	\varepsilon\Ltwox{b_\alpha-\px^\alpha b}^2\leq C(T)\varepsilon+C(T)\E(t).
	\end{equation*}		
	For the term of $ -\varepsilon^2\inx{\left(\frac{\rho}{2}-1\right)\px^\alpha \pt u,\px^\alpha b}$ in \eqref{eq:J311-1}, since neither $\rho$ nor $\rho-2 $ is assumed small, we need to derive some more refined estimates. We substitute $b$ therein using $\eqref{h u}_2$, yielding
	\begin{equation}
	\begin{aligned}
	&-\varepsilon^2\inx{\left(\frac{\rho}{2}-1\right)\px^\alpha \pt u,\px^\alpha b}\\
	=&-\varepsilon^3\inx{\left(\frac{\rho}{2}-1\right)\px^\alpha \pt u,\px^\alpha(\pt u+u\cdot\nabla_xu-\Delta_xu+\nabla_xP) }\\
	=&-\varepsilon^3\intx{\left(\frac{\rho}{2}-1\right)|\px^\alpha \pt u|^2}-\frac{\varepsilon^3}{2}\dt{t}\intx{\left(\frac{\rho}{2}-1\right)|\px^\alpha\nabla_xu|^2}\\
	&+\underbrace{\frac{\varepsilon^3}{4}\intx{\pt\rho |\nabla_x\px^\a u|^2}}_{K_{11}}-\underbrace{\frac{\varepsilon^3}{2}\intx{\px^\a\pt u\cdot \px^\a\nabla_xu\cdot\nabla_x\rho}}_{K_{12}}\\
	&-\underbrace{\varepsilon^3\inx{\left(\frac{\rho}{2}-1\right)\px^\alpha \pt u,\px^\alpha (u\cdot\nabla_xu) }}_{K_{13}}-\underbrace{\varepsilon^3\inx{\left(\frac{\rho}{2}-1\right)\px^\alpha \pt u,\px^\alpha\nabla_xP }}_{K_{14}}.
	\end{aligned}
	\end{equation}
	For $K_{11} $, noting that $\rho=\phi+\rho_0+1 $, we have
	\begin{equation}
	\begin{aligned}
	|K_{11}|=\left|\frac{\varepsilon^3}{4}\intx{\pt\rho |\nabla_x\px^\a u|^2}\right|\leq C(T)\E(t).
	\end{aligned}
	\end{equation}
	For $ K_{12}$, similar to $ K_{11}$,
	\begin{equation}
	|K_{12}|=\left| \frac{\varepsilon^3}{2}\intx{\px^\a\pt u\cdot \px^\a\nabla_xu\cdot\nabla_x\rho}  \right|\leq\delta\varepsilon^3\Lx{\px^\a \pt u}^2+C_\delta (T)\E(t).
	\end{equation}
	For $K_{13} $,
	\begin{equation}
	\begin{aligned}
	|K_{13}|=&\left|\varepsilon^3\inx{\left(\frac{\rho}{2}-1\right)\px^\alpha \pt u,\px^\alpha (u\cdot\nabla_xu) }\right|\\
	\leq&\delta\varepsilon^3\Lx{\px^\a \pt u}^2+C_\delta\varepsilon^2\Hxxx{u}\Hxxx{\n^2 u}^2\\
	\leq&\delta\varepsilon^3\Lx{\px^\a \pt u}^2+C_\delta (T)\E(t).
	\end{aligned}
	\end{equation}
	For  $K_{14}$, integrating by parts yields
	\begin{equation}\label{eq:K14}
	\begin{aligned}
	|K_{14}|=&\left|\varepsilon^3\inx{\left(\frac{\rho}{2}-1\right)\px^\alpha \pt u,\px^\alpha\nabla_xP }\right|\\[5pt]
	=&\frac{1}{2}\left|\varepsilon^3\inx{\nabla_x\rho\cdot\px^\alpha \pt u,\px^\alpha P }\right|\\[5pt]
	\leq&C\varepsilon^3\Linf{\nabla_x\rho}\Lx{\px^\alpha \pt u}\Hxxx{\n_xP}\\[5pt]
	\leq&C\varepsilon^3\Linf{\nabla_x\rho}\Lx{\px^\alpha \pt u}\left(\Hxxx{u\cdot\nabla_xu} +\frac{\Hxxx{b}}{\varepsilon}\right) \\[5pt]
	\leq&\delta\varepsilon^3\Lx{\px^\a \pt u}^2+C_\delta (T)\E(t),
	\end{aligned}
	\end{equation}
	where we apply standard elliptic estimates for $P$ in $\eqref{h u}_2$, namely
	\begin{equation*}
	\Delta_xP=-\div_x(u\cdot\nabla_xu)+\frac{\div_xb}{\varepsilon},
	\end{equation*}
	 and utilize the incompressibility condition $\div_x u=0$.

	Combining the estimates \eqref{eq:J311-1}-\eqref{eq:K14}, we find that
	\begin{equation}\label{eq:J311}
	\begin{aligned}
	J_{311}\leq&-\varepsilon^3\intx{\left(\frac{\rho}{2}-1\right)|\px^\alpha \pt u|^2}-\frac{\varepsilon^3}{2}\dt{t}\intx{\left(\frac{\rho}{2}-1\right)|\px^\alpha\nabla_xu|^2}+C_\delta (T)\varepsilon\\[5pt]
	&+C_\delta (T)\E(t)+C(\delta+\eta)\Big(\varepsilon^3\Ltwox{\px^\alpha \pt u}^2+\varepsilon\Ltwox{b_\alpha}^2+\varepsilon\Lv{(\mathbf{I-P})\partial_{x}^\alpha h}^2\Big).
	\end{aligned}
	\end{equation}

Now, we estimate $J_{312}$, when $\beta<\alpha,\ |\b|\neq 0$, we have
	\begin{equation}\label{eq:J312-1}
	\begin{aligned}
	&\Big|-\frac{1}{2}C^\beta_\alpha \varepsilon^2\inxv{\partial^\beta_x(h(v-u))\cdot \partial^{\alpha-\beta}_x\pt u,\rho_\alpha\sqrt{M_u}+b_\alpha\cdot(v-u)\sqrt{M_u}+(\mathbf{I-P})\partial^\alpha_xh}\Big|\\[5pt]
	\leq&C_\delta (T)\E(t)+C(\delta+\eta)\Big(\varepsilon\Ltwox{b_\alpha}^2+\varepsilon\Lv{(\mathbf{I-P})\partial_{x}^\alpha h}^2\Big) + \delta\varepsilon\Hthreex{\pt \psi}^2+C_\delta (T)\varepsilon.
	\end{aligned}
	\end{equation}		
	When $\beta=\alpha $, we have
	\begin{equation}\label{eq:J312-2}
	\begin{aligned}
   &\Big|-\frac{1}{2} \varepsilon^2\inxv{\partial^\alpha_xh(v-u)\cdot \pt u,\rho_\alpha\sqrt{M_u}+b_\alpha\cdot(v-u)\sqrt{M_u}+(\mathbf{I-P})\partial^\alpha_xh}\Big|\\[5pt]
	%&-\frac{1}{2} \varepsilon^2C^\beta_\alpha\inxv{\partial^{\alpha-\beta}_xh\px^\beta(v-u)\cdot \pt u,\rho_\alpha\sqrt{M_u}+b_\alpha\cdot(v-u)\sqrt{M_u}+(\mathbf{I-P})\partial^\alpha_xh}\\
	=&-\frac{1}{2}\varepsilon^2\inxv{\rho_\alpha \sqrt{M_u}(v-u)\cdot \pt u,b_\alpha\cdot(v-u)\sqrt{M_u}}\\[5pt]
	&-\frac{1}{2}\varepsilon^2\inxv{b_\alpha \cdot(v-u)\sqrt{M_u}(v-u)\cdot \pt u,\rho_\alpha\sqrt{M_u}+(\mathbf{I-P})\px^\alpha h}\\[5pt]
	&-\frac{1}{2}\varepsilon^2\inxv{(\mathbf{I-P})\px^\alpha h(v-u)\cdot \pt u,b_\alpha\cdot(v-u)\sqrt{M_u}+(\mathbf{I-P})\px^\alpha h}\\[5pt]
    \leq&C(\delta+\eta)\Big(\varepsilon\Ltwox{b_\alpha}^2+\varepsilon\Lv{(\mathbf{I-P})\partial_{x}^\alpha h}^2\Big) + \delta\varepsilon\Hthreex{\pt \psi}^2+C_\delta (T)\varepsilon.
	\end{aligned}
	\end{equation}		
	Combining \eqref{eq:J312-1} and \eqref{eq:J312-2}, we obtain
	\begin{equation}\label{eq:J312}
	|J_{312}|\leq C_\delta (T)\E(t)+C(\delta+\eta)\Big(\varepsilon\Ltwox{b_\alpha}^2+\varepsilon\Lv{(\mathbf{I-P})\partial_{x}^\alpha h}^2\Big) + \delta\varepsilon\Hthreex{\pt \psi}^2+C_\delta (T)\varepsilon.
	\end{equation}
	Hence, by further summing \eqref{eq:J311} and \eqref{eq:J312}, we have
	\begin{equation}\label{eq:J31}
	\begin{aligned}
	J_{31}\leq&-\varepsilon^3\intx{\left(\frac{\rho}{2}-1\right)|\px^\alpha \pt u|^2}-\frac{\varepsilon^3}{2}\dt{t}\intx{\left(\frac{\rho}{2}-1\right)|\px^\alpha\nabla_xu|^2}+C_\delta (T)\varepsilon\\[5pt]
	&+C_\delta (T)\E(t)+C(\delta+\eta)\Big(\varepsilon^3\Ltwox{\px^\alpha \pt u}^2 + \varepsilon\Ltwox{b_\alpha}^2+\varepsilon\Lv{(\mathbf{I-P})\partial_{x}^\alpha h}^2\Big)\\[5pt]
    &+\delta\varepsilon\Hthreex{\pt \psi}^2.
	\end{aligned}
	\end{equation}

	For $J_{32} $, we have	
	\begin{equation}\label{eq:J32}
	\begin{aligned}
	|J_{32}|=&\Big|-\varepsilon^2\frac{1}{2}\inxv{\px^\alpha[h (v\cdot\nabla_x u)\cdot(v-u)],\px^\alpha h}\Big|\\[5pt]
	\leq&C(T)\E(t)+C\left|\varepsilon^2\inxv{h (v\cdot\px^\alpha\nabla_x u)\cdot(v-u),\px^\alpha h}\right|\\[5pt]
	\leq&\left|\varepsilon^2\inxv{h ((v-u)\cdot\px^\alpha\nabla_x u)\cdot(v-u),\px^\alpha h}\right|\\[5pt]
	&+\left|\varepsilon^2\inxv{h (u\cdot\px^\alpha\nabla_x u)\cdot(v-u),\px^\alpha h}\right|+C(T)\E(t)\\[5pt]
	\leq&\delta\varepsilon \Big(\Ltwox{\nabla_x^5\psi}^2+\Ltwox{b_\alpha}^2+\Lv{(\mathbf{I-P})\px^\alpha h}^2\Big) + C_\delta (T)\E(t)\\[5pt]
	&+C_\delta (T)\varepsilon+C_\delta(T) \sum\limits_{|\theta|\leq3}\Lv{\px^\theta\mi h}^2.
	\end{aligned}
	\end{equation}	

	For $J_{33}$, we find
	\begin{equation}\label{eq:J33}
	\begin{aligned}
	|J_{33}|=&\big|-\varepsilon^3\inx{\partial^\alpha_x(u\cdot \nabla_xu),\partial^\alpha_x\pt u}\big| \\[5pt]
	\leq&\delta\varepsilon^3\Ltwox{\px^\alpha \pt u}^2+C\varepsilon\Htwox{\pt \psi}^2+C_\delta (T)\E(t)+C_\delta (T)\varepsilon.
	\end{aligned}
	\end{equation}	

	For $J_{34}$, when $|\beta|=1$, integrating by parts yields
    \begin{equation}\label{eq:J34-1-1}
	\begin{aligned}
	&\Big|-\varepsilon C^\beta_\alpha\inxv{\partial^\beta_x\left(\frac{|v-u|^2}{4}\right)\partial^{\alpha-\beta}_xh,\partial^\alpha_xh}\Big|\\
    	=&\Big|\frac{\v}{2} C^\beta_\alpha\inxv{\partial^\beta_x\left(\frac{|v-u|^2}{4}\right),\partial^\b_x(\partial^{\alpha-\beta}_xh)^2}\Big|\\
    =&\Big|\frac{\varepsilon}{2} C^\beta_\alpha\inxv{\px^\b \partial^\beta_x\left(\frac{|v-u|^2}{4}\right)\partial^{\alpha-\beta}_xh,\partial^{\a-\b}_xh}\Big|\\
    \leq&C\v \left|  \inxv{\px^\b u\cdot \px^\b u\,\partial^{\alpha-\beta}_xh,\partial^{\a-\b}_xh}\right|+C\v \left|  \inxv{(v-u)\cdot\px^\b  \px^\b u\,\partial^{\alpha-\beta}_xh,\partial^{\a-\b}_xh}\right|\\
    \leq&C(T)\E(t)+C(T)\v+\underbrace{C\v \left|  \inxv{(v-u)\cdot\px^\b  \px^\b u\,\partial^{\alpha-\beta}_xh,\partial^{\a-\b}_xh}\right|}_{J_{341}},
	\end{aligned}
	\end{equation}
    where we use Lemma \ref{lem:hx} in the last inequality above. Note that, for $J_{341} $, using the Macro-Micro decomposition, we have
   \begin{equation}\label{eq:J34-1-2}
    \begin{aligned}
    J_{341}=&C\v \left|  \inxv{(v-u)\cdot\px^\b  \px^\b u\,\partial^{\alpha-\beta}_xh,\partial^{\a-\b}_xh}\right|\\
    \leq&C\v \left|  \inxv{(v-u)\cdot\px^\b  \px^\b u,\partial^{\a-\b}_x\left((a+1)\sqrt{M_u}+b\cdot(v-u)\sqrt{M_u}+\mi g\right)^2}\right|\\
    \leq&\delta\varepsilon(\Ltwox{b_\alpha}^2+\Lv{(\mathbf{I-P})\px^\alpha h}^2) +C_\delta (T)\E(t)+C_\delta (T)\varepsilon+C_\delta (T)\sum\limits_{|\theta|\leq 3}\Lv{\px^\theta\mi h}^2,
    \end{aligned}
   \end{equation}
    where we use Lemma \ref{lem:b-b_a} in the last inequality above.

    Therefore, for $J_{34} $, when $2\leq|\beta|\leq 3 $, similar to $J_{341} $ , integrating by parts, then for $|\a'|=1$,  we have
  \begin{equation}\label{eq:J34-2}
	\begin{aligned}
	&\Big|-\varepsilon C^\beta_\alpha\inxv{\partial^\beta_x\left(\frac{|v-u|^2}{4}\right)\partial^{\alpha-\beta}_xh,\partial^\alpha_xh}\Big|\\[5pt]
    =&\Big|\varepsilon C^\beta_\alpha\inxv{\px^{\a'}\left( \partial^\beta_x\left(\frac{|v-u|^2}{4}\right)\partial^{\alpha-\beta}_xh\right),\partial^{\a-\a'}_xh}\Big|\\[5pt]
     \leq&\delta\varepsilon(\Ltwox{b_\alpha}^2+\Lv{(\mathbf{I-P})\px^\alpha h}^2) +C_\delta (T)\E(t)+C_\delta (T)\varepsilon+C_\delta (T)\sum\limits_{|\theta|\leq 3}\Lv{\px^\theta\mi h}^2.
	\end{aligned}
	\end{equation}		
 When $|\b|=|\a|=4 $, we have
	 \begin{equation}\label{eq:J34-3}
	\begin{aligned}
     &\Big|-\varepsilon \inxv{\partial^\a_x\left(\frac{|v-u|^2}{4}\right)\,h,\partial^\alpha_xh}\Big|\\[5pt]
    =&\Big|\varepsilon \inxv{\px^{\a'}\left( \partial^\a_x\left(\frac{|v-u|^2}{4}\right)h\right),\partial^{\a-\a'}_xh}\Big|\\[5pt]
     \leq&\delta\varepsilon(\Ltwox{b_\alpha}^2+\Lv{(\mathbf{I-P})\px^\alpha h}^2) +C_\delta (T)\E(t)+C_\delta (T)\varepsilon\\[5pt]
     &+C_\delta (T)\sum\limits_{|\theta|\leq 3}\Lv{\px^\theta\mi h}^2+\delta\varepsilon\Lx{\nabla^5_x\psi}^2.
	\end{aligned}
	\end{equation}	
Combining the estimates \eqref{eq:J34-1-1}-\eqref{eq:J34-3}, we get
    \begin{multline}\label{eq:J34}
      J_{34} \leq \delta \varepsilon \Big(\Ltwox{b_\alpha}^2+\Lv{(\mathbf{I-P})\px^\alpha h}^2\Big) +C_\delta (T)\E(t)+C_\delta (T)\varepsilon\\
      +C_\delta (T)\sum\limits_{|\theta|\leq 3}\Lv{\px^\theta\mi h}^2+\delta\varepsilon\Lx{\nabla^5_x\psi}^2.
    \end{multline}

	Thus, combining the estimates \eqref{eq:J31}, \eqref{eq:J32}, \eqref{eq:J33} and \eqref{eq:J34} yields
	\begin{equation*}
	\begin{aligned}
	&\frac{1}{2}\dt{t}\left(\varepsilon^2\Ltwoxv{\px^\alpha h}^2+\frac{\varepsilon^3}{2}\Ltwox{\sqrt{\rho}\px^\alpha\nabla_xu}^2\right)\\[5pt]
	&+\frac{\varepsilon^3}{2}\Ltwox{\sqrt{\rho}\px^\alpha \pt u}^2+\varepsilon\Ltwox{b_\alpha}^2+\varepsilon c\Lv{(\mathbf{I-P})\partial_{x}^\alpha h}^2\\[5pt]
	\leq&C_\delta (T)\E(t)+\delta\varepsilon\Hthreex{\pt \psi}^2+C_\delta (T)\varepsilon + C\delta\varepsilon \left( \Ltwox{b_\alpha}^2+\Lv{(\mathbf{I-P})\px^\alpha h}^2 \right)\\[5pt]
	&+C_\delta (T)\sum\limits_{|\theta|\leq3}\Lv{\px^\theta\mi h}^2+\delta\varepsilon^3\Ltwox{\px^\alpha \pt u}^2+\delta\varepsilon\Ltwox{\nabla_x^5\psi}^2.
	\end{aligned}
	\end{equation*}		
By further choosing $\delta$ and $\eta$ small enough to satisfy
\begin{equation}\label{eq:def eta}
C(\delta+\eta)\leq \frac{c}{4},
\end{equation}
and integrating over the time interval $(0,t)$, we can finally obtain \eqref{eq:hH4}.
\end{proof}
	
In what follows, we will combine the macroscopic error estimates in Corollary \ref{coro:ma total} with those for $g$ (or $h$) in Lemma \ref{lem:hH4}, and thereby prove the Proposition \ref{lem:ma g}.
\begin{proof}[{\bf Proof of Proposition \ref{lem:ma g}}]
Note that, for $1\leq|\alpha|\leq4 $, in the expressions below, $\beta$ and $\theta$ denote that the summation is taken over all multi-indices satisfying $1\leq|\beta|+|\theta|\leq |\a|$ and $\beta,\ \theta \leq \alpha $. Then, we have
		\begin{equation*}
		\begin{aligned}
		\partial_{x}^\alpha((\mathbf{I-P})h)=&\partial_{x}^\alpha h-\partial_{x}^\alpha\left(\intv{h\sqrt{M_u}}\sqrt{M_u}\right)-\partial_{x}^\alpha\left(\intv{(v-u)h\sqrt{M_u}}\cdot(v-u)\sqrt{M_u}\right)\\
		=&(\mathbf{I-P})\partial_{x}^\alpha h-\left(\intv{\partial_{x}^{\alpha-\beta-\theta} h\,\partial_{x}^{\beta}\sqrt{M_u}}\,\partial_{x}^{\theta}\sqrt{M_u}\right)\\
		&-\left(\intv{\partial_{x}^{\alpha-\beta-\theta}h\, \partial_{x}^{\beta}[(v-u)\sqrt{M_u}]}\cdot\partial_{x}^{\theta}[(v-u)\sqrt{M_u}]\right),
		\end{aligned}
		\end{equation*}
		which yields that	
		\begin{equation}\label{eq:mi H4}
		\begin{aligned}
		\varepsilon\Lv{\partial_{x}^\alpha(\mathbf{I-P})h}^2\leq&\varepsilon\Lv{(\mathbf{I-P})\partial_{x}^\alpha h}^2+C(T)\mathbb{E}(t)+C(T)\varepsilon\HxxxLv{(\mathbf{I-P})h}^2+C(T)\varepsilon\\[5pt]
		\leq&\varepsilon\Lv{(\mathbf{I-P})\partial_{x}^\alpha h}^2+C(T)\mathbb{E}(t)+C(T)\varepsilon.
		\end{aligned}
		\end{equation}
Using $\mi g=\mi h $ in \eqref{eq:relation-g-h}, we have, for all $|\a|\leq 4 $,
		\begin{equation*}
		\varepsilon\Lv{\partial_{x}^\alpha(\mathbf{I-P})h}^2=\varepsilon\Lv{\partial_{x}^\alpha(\mathbf{I-P})g}^2.
		\end{equation*}
		Finally, by substituting \eqref{eq:mi H4} into \eqref{eq:ma total}, combining with Lemma \ref{lem:hH4}, and choosing $\delta$ and $\varepsilon$ sufficiently small, we can obtain \eqref{eq:ma g}.		
	\end{proof}

%----------------------------
\subsection{Microscopic estimates}
\label{subsec:micro-estimate}

In this section, we derive the energy estimate for the microscopic part $(\mathbf{I-P})g$.
To this end, we first derive the microscopic system of $(\mathbf{I-P})g$ by substituting the decomposition $\eqref{decom}$ into $\eqref{g}_1$:
	\begin{equation}
	\begin{aligned}
	\label{i-p}
	&\pt((\mathbf{I-P})g)+v\cdot \nabla_x((\mathbf{I-P})g)+\frac{\mathbf{L}((\mathbf{I-P})g)}{\varepsilon}=-R_{mi},
	\end{aligned}
	\end{equation}
	where
	\begin{equation*}
	\begin{aligned}
	R_{mi}=&-b\cdot \pt u\sqrt{M_u}+b\cdot(v-u)\pt u\cdot(v-u)\sqrt{M_u}+(\rho(v-u)\cdot \nabla_x u)\cdot(v-u)\sqrt{M_u}\\[5pt]
	&+((v-u)\cdot\nabla_xb)\cdot(v-u)\sqrt{M_u}-(v\cdot\nabla_xu)\cdot b\sqrt{M_u}-\mathrm{div}_xb\sqrt{M_u}\\[5pt]
	&+b\cdot(v-u)(v\cdot \nabla_xu)\cdot(v-u)\sqrt{M_u}+\frac{1}{2}(\mathbf{I-P})g(v-u)\cdot \partial_t u\\[5pt]
	&+\frac{1}{2}(\mathbf{I-P})g(v\cdot \nabla_xu)\cdot(v-u)-(b\cdot \nabla_xu)\cdot(v-u)\sqrt{M_u}\\[5pt]
	&-\intv{(v-u)v\cdot \nabla_x((\mathbf{I-P})g\sqrt{M_u})}\cdot (v-u)\sqrt{M_u}.
	\end{aligned}
	\end{equation*}

	\begin{lemma}\label{lem:mi H3x}
			Under the conditions of Theorem \ref{thm1} and the \textit{a priori} assumption \eqref{assumption}, for $t\in [0,T_1] $, we have
   		\begin{equation}\label{eq:mi H3x}
   \begin{aligned}
   &\|(\mathbf{I-P})g\|_{H^2_x L^2_v}^2+\v\|(\mathbf{I-P})g\|_{H^3_x L^2_v}^2\\[5pt]
   & +\int_0^t\Big(\sum\limits_{|\a|\leq 2}\frac{\|\p_x^\a(\mathbf{I-P})g\|_\nu^2}{\v}+\sum\limits_{|\theta|=3 }\|\p_x^\theta(\mathbf{I-P})g\|_\nu^2\Big)\mathrm{d}s\\
   \leq&C(T)\int_0^t \mathbb{E}(s)\mathrm{d}s+C(T)\mathbb{E}(0)+C(T)\varepsilon\\
&+C(T)\int_0^t\Big[\v^2\|\p_t\psi\|_{H^3_x}^2+\v^2\sum\limits_{|\lambda|=4}(\|b_\lambda\|_{L^2_x}^2+\|(\mathbf{I-P})\p_x^\lambda h\|_\nu^2)\Big]\mathrm{d}s,
   \end{aligned}
   \end{equation}	
     where the energy functional $\E(t)$ is given in \eqref{defE} and $C(T) > 1$ is a constant only depending on $T$.
	\end{lemma}
	
	\begin{proof}
		Applying $\p^\a_x$ with $0\leq|\a|\leq 2 $ to $\eqref{i-p}$, multiplying by $\p^\a_x(\mathbf{I-P})g$, and then integrating over $x,\ v \in \mathbb{R}^3$, we obtain
		\begin{equation}\label{i-pgh2}
		\begin{aligned}
		\frac{1}{2}\dt{t}\|\px^\a(\mathbf{I-P})g\|_{L^2_x L^2_v}^2+\frac{c \|(\mathbf{I-P})\p_x^\a (\mathbf{I-P})g\|_\nu^2}{\v}\le \sum\limits_{i=1}^{10}B_i,			
		\end{aligned}
		\end{equation}
		where, we have used Lemma \ref{lemma A2} and
			\begin{equation*}
		\begin{aligned}
		\p^\a_x\mathbf{L}(\mathbf{I-P})g=&\mathbf{L}(\p^\a_x(\mathbf{I-P})g)+\sum\limits_{\lambda\leq\a,|\lambda|\neq 0}C^{\lambda}_{\a} \left[\p^\lambda_x \left(\frac{|v-u|^2}{4}\right) \p^{\a-\lambda}_x(\mathbf{I-P})g \right].
		\end{aligned}
		\end{equation*}
		
		 Throughout the rest of the proof, the summation notation $\sum\limits_{1 \leq |\b| \leq |\a|}$ will be omitted for brevity. The source terms $ B_i$ are defined as
		\begin{equation*}
		\begin{aligned}
		B_1=&\frac{1}{\v}|C^{\b}_{\a}\inxv{\p^\b_x \left(\frac{|v-u|^2}{4} \right)\p^{\a-\b}_x(\mathbf{I-P})g,\p^\a_x(\mathbf{I-P})g}|,\\[5pt]
		B_2=&|\inxv{\p_{x}^\a(b\cdot \p_t u\M),\p^\a_x(\mathbf{I-P})g}|\\[3pt]
		&+\Big|\inxv{\p_x^\a\Big\{\big[b\cdot(v-u)\big]\big[\p_t u\cdot(v-u)\M\big]\Big\},\p^\a_x(\mathbf{I-P})g}\Big|,\\[5pt]
		B_3=&\Big|\inxv{\p_x^\a\Big\{\big[a(v-u)\cdot \n_x u\big]\cdot(v-u)\M\Big\},\p^\a_x(\mathbf{I-P})g}\Big|,\\[5pt]
		B_4=&\Big|\inxv{\p_x^\a\Big\{\big[(v-u)\cdot\n_xb\big]\cdot(v-u)\M\Big\},\p^\a_x(\mathbf{I-P})g}\Big|,\\[5pt]
		B_5=&\big|\inxv{\p_x^\a\big[(v\cdot\n_xu)\cdot b\M\big],\p^\a_x(\mathbf{I-P})g}\big|\\[3pt]
		&+\big|\inxv{\p_x^\a\big[b\cdot(v-u)(v\cdot \n_xu)\cdot(v-u)\M\big],\p^\a_x(\mathbf{I-P})g}\big|,\\[5pt]
		B_6=&\Big|\inxv{\p_x^\a\Big\{\big[(v-u)\cdot \n_x u\big]\cdot(v-u)\M\Big\},\p^\a_x(\mathbf{I-P})g}\Big|,\\[5pt]
		B_7=&\big|\inxv{\p_x^\a \left[\frac{1}{2}(\mathbf{I-P})g(v-u)\cdot \p_t u\right],\p^\a_x(\mathbf{I-P})g}\big|,\\[5pt]
		B_8=&\big|\inxv{\p_x^\a\left[\frac{1}{2}(\mathbf{I-P})g(v\cdot \n_xu)\cdot(v-u)\right],\p^\a_x(\mathbf{I-P})g}\big|,\\[5pt]
		B_9=&\big|\inxv{\p_x^\a\big[\mathrm{div}_xb\M+(b\cdot \n_xu)\cdot(v-u)\M\big],\p^\a_x(\mathbf{I-P})g}\big|,\\[5pt]
		B_{10}=&\big|\inxv{\p_x^\a \left[\l v\cdot \n_x\big[(\mathbf{I-P})g\M\big],v-u\r_v\cdot (v-u)\M\right],\p^\a_x(\mathbf{I-P})g}\big|.
		\end{aligned}
		\end{equation*}
		For $B_1$, when $|\alpha|=0$, this term does not appear, when $|\alpha|=1$, we get
		\begin{equation}\label{eq:B1}
		\begin{aligned}
		B_1=&\frac{1}{\v}|C^\b_\a\inxv{\p^\b_x \left(\frac{|v-u|^2}{4}\right)\p^{\a-\b}_x (\mathbf{I-P})g,\p^\a_x(\mathbf{I-P})g}|\\[3pt]
		\leq&C(T)\frac{1}{\v}\|\p_x^\a(\mathbf{I-P})g\|_\nu\|\p_x^{\a-\a'}(\mathbf{I-P})g\|_{L^2_xL^2_v}\\
		\leq&\delta\frac{\|\p_x^\a(\mathbf{I-P})g\|_\nu^2}{\v}+C_\delta(T)\frac{\|(\mathbf{I-P})g\|^2_{L^2_xL^2_v}}{\v}.
		\end{aligned}
		\end{equation}
		 When $|\alpha|=2$, we have
			\begin{equation}\label{eq:B1-1}
		\begin{aligned}
		B_1=&\frac{1}{\v}|C^\b_\a\inxv{\p^\b_x \left(\frac{|v-u|^2}{4}\right)\p^{\a-\b}_x (\mathbf{I-P})g,\p^\a_x(\mathbf{I-P})g}|\\[3pt]
		\leq&C(T)\frac{1}{\v}\|\p_x^\a(\mathbf{I-P})g\|_\nu(\|\p_x^{\a-\a'}(\mathbf{I-P})g\|_{L^2_xL^2_v}+\|\p_x^{\a-\a''}\Big( \mathbf{I-P})g\|_{L^2_xL^2_v} \Big)\\
		\leq&\delta\frac{\|\p_x^\a(\mathbf{I-P})g\|_\nu^2}{\v}+C_\delta (T)\frac{\|(\mathbf{I-P})g\|_{H^1_xL^2_v}^2}{\v}.
		\end{aligned}
		\end{equation}	
		For $B_2$, noting that $\p_t u=\p_t\psi+\p_tu_0$,
		\begin{equation}\label{eq:B2}
		\begin{aligned}
		B_2=&|\inxv{\p_x^\a(b\cdot \p_t u\M),\p^\a_x(\mathbf{I-P})g}|\\[3pt]
		& \qquad \qquad +|\inxv{\p_x^\a\big[b\cdot(v-u)\p_t u\cdot(v-u)\M\big],\p^\a_x(\mathbf{I-P})g}|\\[5pt]
		\leq&\delta\frac{\|\p_x^\a(\mathbf{I-P})g\|_\nu^2}{\v}+ C_\delta(T)\mathbb{E}(t).
		\end{aligned}
		\end{equation}
		For $B_3$, noting that $a=\phi+\rho_0 $,
		\begin{equation}\label{eq:B3}
		\begin{aligned}	
		B_3=&\Big|\inxv{\p_x^\a\Big\{\big[a(v-u)\cdot \n_x u\big]\cdot(v-u)\M\Big\},\p^\a_x(\mathbf{I-P})g}\Big|\\
		\leq&\delta\frac{\|\p_x^\a(\mathbf{I-P})g\|_\nu^2}{\v}+ C_\delta (T)\mathbb{E}(t)+ C_\delta (T)\v.
		\end{aligned}
		\end{equation}
		For $B_4$, we have:
		\begin{equation}\label{eq:B4}
		\begin{aligned}	
		B_4=&\Big|\inxv{\p_x^\a\Big\{\big[(v-u)\cdot\n_xb\big]\cdot(v-u)\M\Big\},\p^\a_x(\mathbf{I-P})g}\Big|\\
		\leq&\delta\frac{\|\p_x^\a(\mathbf{I-P})g\|_\nu^2}{\v}+ C_\delta(T)\mathbb{E}(t)+ C_\delta(T)\v.
		\end{aligned}
		\end{equation}
		For $B_5$, using $ u=\psi+u_0$, we obtain:
		\begin{equation}
		\begin{aligned}	
		B_5=&\big|\inxv{\p_x^\a\big[(v\cdot\n_xu)\cdot b\M\big],\p^\a_x(\mathbf{I-P})g}\big|\\[3pt]
		& \qquad \qquad +|\inxv{\p_x^\a\big[b\cdot(v-u)(v\cdot \n_xu)\cdot(v-u)\M\big],\p^\a_x(\mathbf{I-P})g}|\\[5pt]
		\leq&\delta\frac{\|\p_x^\a(\mathbf{I-P})g\|_\nu^2}{\v}+C_\delta (T)\mathbb{E}(t)+C_\delta (T)\v.
		\end{aligned}
		\end{equation}
		For $B_{6}$,  using $ u=\psi+u_0$ again, we have
		\begin{equation}
		\begin{aligned}
		B_{6}=&\Big|\inxv{\p_x^\a\Big\{\big[(v-u)\cdot \n_x u\big]\cdot(v-u)\M\Big\},\p^\a_x(\mathbf{I-P})g}\Big|\\
		\leq&\delta\frac{\|\p_x^\a(\mathbf{I-P})g\|_\nu^2}{\v}+C_\delta (T)\mathbb{E}(t)+C_\delta (T)\v.
		\end{aligned}
		\end{equation}
		For $B_7$, noting that $\p_t u=\p_t\psi+\p_tu_0$, we have
		\begin{equation}
		\begin{aligned}
		B_7=&\big|\inxv{\p_x^\a\left[\frac{1}{2}(\mathbf{I-P})g(v-u)\cdot \p_t u\right],\p^\a_x(\mathbf{I-P})g}\big|\\[5pt]	
		\leq&\delta\frac{\|\p_{x}^\a(\mathbf{I-P})g\|_\nu^2}{\v}+C_\delta (T)\mathbb{E}(t)+C_\delta (T)\v.
		\end{aligned}
		\end{equation}
		For $B_8$, using $u=\psi+u_0$, we get
		\begin{equation}
		\begin{aligned}
		B_8=&\big|\inxv{\p_x^\a\left[\frac{1}{2}(\mathbf{I-P})g(v\cdot \n_xu)\cdot(v-u)\right],\p^\a_x(\mathbf{I-P})g}\big|\\
		\leq&\delta\frac{\|\p_x^\a(\mathbf{I-P})g\|_\nu^2}{\v}+C_\delta (T)\mathbb{E}(t)+C_\delta (T)\v.
		\end{aligned}
		\end{equation}
		For $B_9$, we have the similar estimate as for $B_4-B_6$,
		\begin{equation}
		B_9\leq\delta\frac{\|\p_x^\a(\mathbf{I-P})g\|_\nu^2}{\v}+ C_\delta (T)\mathbb{E}(t)+ C_\delta (T)\v.
		\end{equation}
		For $B_{10} $, we have
		\begin{equation}\label{eq:B10}
		\begin{aligned}
		B_{10}=&|\inxv{\p_x^\a(\l v\cdot \n_x((\mathbf{I-P})g\M),v-u\r_v\cdot (v-u)\M),\p^\a_x(\mathbf{I-P})g}|\\[5pt]
		\leq&\delta\frac{\|\p_x^\a(\mathbf{I-P})g\|_\nu^2}{\v}+C_\delta (T)\mathbb{E}(t)+C_\delta (T)\v+C_\delta (T)\v\|(\mathbf{I-P})g\|_{H^3_x L^2_v}^2\\
		\leq&\delta\frac{\|\p_x^\a(\mathbf{I-P})g\|_\nu^2}{\v}+C_\delta (T)\mathbb{E}(t)+C_\delta (T)\v.
		\end{aligned}
		\end{equation}
		Plugging \eqref{eq:B1}-\eqref{eq:B10} into \eqref{i-pgh2}, using \eqref{eq:x mi} in Lemma \ref{lem:property mi}, choosing $\delta$ small enough, we obtain
		\begin{equation}\label{(i-p)g2-0}
		\dt{t}\Lxv{(\mathbf{I-P})g}^2+C_1\frac{\|(\mathbf{I-P})g\|_\nu^2}{\v} \leq C(T)\mathbb{E}(t)+C(T)\v,	
		\end{equation}
		and the first-order spatial derivative estimates
		\begin{equation}\label{(i-p)g2-1}
		\begin{aligned}
	&\dt{t}\Lxv{\nabla_x(\mathbf{I-P})g}^2+C_1\sum\limits_{|\a|=1}\frac{\|\p_x^\a (\mathbf{I-P})g\|_\nu^2}{\v}\\
	 \leq &C(T)\frac{\|(\mathbf{I-P})g\|_\nu^2}{\v}+ C(T)\mathbb{E}(t)+C(T)\v,	
			\end{aligned}
	\end{equation}	
	and the second-order spatial derivative estimates
			\begin{equation}\label{(i-p)g2-2}
			\begin{aligned}
		&\dt{t}\Lxv{\nabla_x^2(\mathbf{I-P})g}^2+C_1\sum\limits_{|\a|=2}\frac{\|\p_x^\a (\mathbf{I-P})g\|_\nu^2}{\v}\\
		 \leq& C(T)\sum\limits_{|\a|\leq 1}\frac{\|\p_x^\a (\mathbf{I-P})g\|_\nu^2}{\v}+ C(T)\mathbb{E}(t)+C(T)\v.
				\end{aligned}	
		\end{equation}
Combining  \eqref{(i-p)g2-0}-\eqref{(i-p)g2-2}, we obtain
	  	\begin{equation}\label{(i-p)g2}
	  \begin{aligned}
	  &4\left(\frac{C(T)}{C_1}+1\right)^2\dt{t}\Lxv{(\mathbf{I-P})g}^2+\left(\frac{2C(T)}{C_1}+1\right)\dt{t}\Lxv{\nabla_x(\mathbf{I-P})g}^2\\
	  &+\dt{t}\Lxv{\nabla^2_x(\mathbf{I-P})g}^2+C_1\sum\limits_{|\a|\leq 2}\frac{\|\p_x^\a (\mathbf{I-P})g\|_\nu^2}{\v}\\
	  \leq& C(T)\mathbb{E}(t)+C(T)\v.
	  \end{aligned}	
	  \end{equation}	
	Furthermore, applying $\p^\theta_x$ with $|\theta|=3$ to \eqref{i-p}, multiplying by $\v\p^\theta_x(\mathbf{I-P})g$, and integrating over $x, v \in \mathbb{R}^3$, we obtain, similarly to the preceding estimates,
	\begin{equation}\label{eq:123}
	\begin{aligned}
	&\sum\limits_{|\theta|=3}\v\dt{t}\|\p^\theta_x(\mathbf{I-P})g\|_{L^2_x L^2_v}^2+C_1\sum\limits_{|\theta|=3 }\|\p_x^\theta (\mathbf{I-P})g\|_\nu^2\\[3pt]
	&\leq C(T)\mathbb{E}(t)+C(T)\v+C\v^2\|\p_t\psi\|_{H^3_x}^2\\[3pt]
	&\quad+C(T)\v^2\|\n_x^4(\mathbf{I-P})g\|_{L^2_x L^2_v}^2+C(T)\v^2\|\n^4_x b\|_{L^2_x}^2+C(T)\sum\limits_{|\a|\leq 2}\|\p_x^\a(\mathbf{I-P})g\|_\nu^2.
	\end{aligned}
	\end{equation}
	In addition, by an argument similar to that used for \eqref{eq:mi H4}, we have the following bound:
	\begin{equation}\label{eq:11}
	\begin{aligned}
	&\v\|\n_x^4(\mathbf{I-P})g\|_{L^2_x L^2_v}^2+\v\|\n^4_x b\|_{L^2_x}^2\\[5pt]
	\leq & C\v\sum\limits_{|\lambda|=4 }(\|b_\lambda\|_{L^2_x}^2+\|(\mathbf{I-P})\p_x^\lambda h\|_\nu^2)+C (T) \sum\limits_{|\a|\leq 3}\|\p_x^\a(\mathbf{I-P})g\|_\nu^2\\[5pt]
    &+C(T)\E(t)+C(T)\v.
	\end{aligned}
	\end{equation}
	Substituting \eqref{eq:11} into \eqref{eq:123}, we have
	\begin{equation}\label{(i-p)g3}
	\begin{aligned}
	&\sum\limits_{|\theta|=3}\v\dt{t}\|\p^\theta_x(\mathbf{I-P})g\|_{L^2_x L^2_v}^2+C_1\sum\limits_{|\theta|=3 }\|\p_x^\theta (\mathbf{I-P})g\|_\nu^2\\[5pt]
	\leq& C(T)\mathbb{E}(t)+C(T)\v+C(T)\v^2\|\p_t \psi\|_{H^3_x}^2\\[5pt]
	&+C(T)\v^2\sum\limits_{|\lambda|=4}(\|b_\lambda\|_{L^2_x}^2+\|(\mathbf{I-P})\p_x^\lambda h\|_\nu^2)+C(T)\sum\limits_{|\a|\leq 2}\|\p_x^\a(\mathbf{I-P})g\|_\nu^2.
	\end{aligned}
	\end{equation}
		
	Finally, adding \eqref{(i-p)g2} and \eqref{(i-p)g3} together, choosing $\varepsilon$ small enough, we obtain
		\begin{equation}
		\begin{aligned}
		 &4\left(\frac{C(T)}{C_1}+1\right)^2\dt{t}\Lxv{(\mathbf{I-P})g}^2+\left(\frac{2C(T)}{C_1}+1\right)\dt{t}\Lxv{\nabla_x(\mathbf{I-P})g}^2\\[5pt]
		&+\dt{t}\Lxv{\nabla^2_x(\mathbf{I-P})g}^2+\v\dt{t}\|\n_x^3(\mathbf{I-P})g\|_{L^2_x L^2_v}^2\\[5pt]
		&+\frac{C_1}{2}\sum\limits_{|\a|\leq 2}\frac{\|\p_x^\a(\mathbf{I-P})g\|_\nu^2}{\v}+\frac{C_1}{2}\sum\limits_{|\theta|=3 }\|\p_x^\theta (\mathbf{I-P})g\|_\nu^2\\[5pt]
		\leq &C(T)\mathbb{E}(t)+C(T)\v
		+C(T)\v^2\|\p_t\psi\|_{H^3_x}^2  +C(T)\v^2\sum\limits_{|\lambda|=4}(\|b_\lambda\|_{L^2_x}^2+\|(\mathbf{I-P})\p_x^\lambda h\|_\nu^2),
		\end{aligned}
		\end{equation}	
		which yields \eqref{eq:mi H3x}. This completes the proof.
	\end{proof}
	
Now, we first summarize all the estimates except those involving the mixed spatial-velocity derivatives of $\mi$. The bounds for these mixed derivatives will be established in the next lemma, thereby completing the overall energy estimates.
	\begin{corollary}\label{coro:ma g mix}	Under the conditions of Theorem \ref{thm1} and the \textit{a priori} assumption \eqref{assumption}, for $t\in [0,T_1] $, we have
    \begin{equation}\label{eq:ma g mix}
         \begin{aligned}
         &\Hxxx{(\phi,\psi,b)}^2+\frac{\Hxx{b}^2}{\varepsilon}+\varepsilon\Hxx{\p_t\psi}^2+\varepsilon\Hxxx{\nabla_x\psi}^2+\varepsilon^2\HxxxLv{\nabla_xh}^2+\varepsilon^3\Hthreex{\nabla_x^2 u}^2\\
          &+\|(\mathbf{I-P})g\|_{H^2_x L^2_v}^2+\v\|(\mathbf{I-P})g\|_{H^3_x L^2_v}^2\\[5pt]
         & +\int_0^t\Big(\sum\limits_{|\a|\leq 2}\frac{\|\p_x^\a(\mathbf{I-P})g\|_\nu^2}{\v}+\sum\limits_{|\theta|=3 }\|\p_x^\theta(\mathbf{I-P})g\|_\nu^2\Big)\mathrm{d}s\\
         &+\int_0^t\left( \Hxxx{\nabla_x\psi}^2+\v\|\n_x\psi\|_{H^4_x}^2+\frac{\Hxxx{b}^2}{\varepsilon}+\Hxx{\p_tb}^2+\Hxx{\p_t\psi}^2+\varepsilon\Hxx{\nabla_x\p_t\psi}^2\right)\d s\\
         &+\intt{\left(\varepsilon^3\Hthreex{\nabla_x\pt u}^2+\sum\limits_{1\leq|\alpha|\leq4}\varepsilon(\Ltwox{b_\alpha}^2+\Lv{\mi \px^\alpha h}^2)\right)}{s}\\
         \leq&C(T)\E(0)+C(T)\varepsilon+C(T)\intt{\E(s)}{s},
         \end{aligned}
         \end{equation}
           where the energy functional $\E(t)$ is given in \eqref{defE} and $C(T) > 1$ is a constant only depending on $T$.
	\end{corollary}
	\begin{proof}
  Combining \eqref{eq:ma g} and \eqref{eq:mi H3x} and choosing $\delta$ and $\varepsilon$ sufficiently small directly yields the corollary.
	\end{proof}
	
\begin{lemma}\label{lem:mi xv}	Under the conditions of Theorem \ref{thm1} and the \textit{a priori} assumption \eqref{assumption}, for $t\in [0,T_1] $, we have
		\begin{equation}\label{eq:mi xv}
		\begin{aligned}
		&\sum\limits_{0\leq |\a|\leq 2,1\leq |\b|\leq 2}\Big(\|\p^\a_x\p^\b_v(\mathbf{I-P})g\|_{L^2_x L^2_v}^2+\int_0^t\frac{\|\p^\a_x\p^\b_v(\mathbf{I-P})g\|_\nu^2}{\v}\mathrm{d}s\Big)\\
		\leq& C(T)\mathbb{E}(0)+C(T)\varepsilon+C(T)\int_0^t \E(s)\mathrm{d}s,
		\end{aligned}
		\end{equation}
          where the energy functional $\E(t)$ is given in \eqref{defE} and $C(T) > 1$ is a constant only depending on $T$.
	\end{lemma}
\begin{proof}
	Applying $\partial^\alpha_x\partial^\beta_v$ with $0\leq|\alpha|\leq 2,\ 1\leq|\beta|\leq 2$ to $\eqref{i-p}$, multiplying the resulting identity by $\partial^\alpha_x\partial^\beta_v(\mathbf{I-P})g$, and integrating over $x,v \in \mathbb{R}^3$, we obtain
	\begin{equation}\label{eq:+remainder}
	\frac{1}{2}\dt{t} \|\partial^\alpha_x\partial^\beta_v(\mathbf{I-P})g\|_{L^2_x L^2_v}^2+\frac{c\|\mi \partial^\alpha_x\partial^\beta_v(\mathbf{I-P})g\|_\nu^2}{\v}
	\lesssim \sum\limits_{i=1}^{12}J_i.
	\end{equation}
	Here, we have utilized the expansion
	\begin{equation*}
	\begin{aligned}
	\partial^\alpha_x\partial^\beta_v\mathbf{L}(\mathbf{I-P})g=&\mathbf{L}\big(\partial^\alpha_x\partial^\beta_v(\mathbf{I-P})g\big)+\sum\limits_{\lambda\leq \alpha,|\lambda|\neq 0} C^{\lambda}_{\alpha}\partial^\beta_v \left[\partial^\lambda_x \left(\frac{|v-u|^2}{4}\right) \partial^{\alpha-\lambda}_x(\mathbf{I-P})g \right]\\[5pt]
	&+\sum\limits_{\theta\leq \beta,|\theta|\neq 0} C^{\theta}_{\beta}\partial^\theta_v \left(\frac{|v-u|^2}{4}\right) \partial^{\beta-\theta}_v\partial^{\alpha}_x(\mathbf{I-P})g,
	\end{aligned}
	\end{equation*}
	where $\lambda$ and $\theta$ stand for multi-indices satisfying $\lambda \leq \alpha,\ |\lambda|\neq 0$ and $\theta \leq \beta,\ |\theta|\neq 0$, respectively.
	
	The quantities $J_i$ are defined as follows:
	\begin{equation*}
	\begin{aligned}
	J_1=&\ \frac{1}{\v}\Big|C^{\lambda}_{\alpha}\Big\langle\partial^\beta_v \left[\partial^\lambda_x \left(\frac{|v-u|^2}{4}\right) \partial^{\alpha-\lambda}_x(\mathbf{I-P})g \right],\partial^\alpha_x\partial^\beta_v(\mathbf{I-P})g\Big\rangle_{x,v}\Big|,\\[5pt]
	J_2=&\ \frac{1}{\v}\Big|C^{\theta}_{\beta}\Big\langle\partial^\theta_v \left(\frac{|v-u|^2}{4}\right) \partial^{\beta-\theta}_v\partial^\alpha_x(\mathbf{I-P})g,\partial^\alpha_x\partial^\beta_v(\mathbf{I-P})g\Big\rangle_{x,v}\Big|,\\[5pt]
	J_3=&\ \Big|\Big\langle\partial^\alpha_x\partial^\beta_v\big[v\cdot \nabla_x(\mathbf{I-P})g\big],\partial^\alpha_x\partial^\beta_v(\mathbf{I-P})g\Big\rangle_{x,v}\Big|,\\[5pt]
	J_4=&\ \Big|\Big\langle\partial^\alpha_x\partial^\beta_v\Big\{\big\langle v\cdot \nabla_x\big[(\mathbf{I-P})g\, \M \big],v-u\big\rangle_v\cdot (v-u)\Big\},\partial^\alpha_x\partial^\beta_v(\mathbf{I-P})g\Big\rangle_{x,v}\Big|\\
    	J_5=&|\inxv{\p_{x}^\a\p_v^\b(b\cdot \p_t u\M),\p^\a_x\p_v^\b(\mathbf{I-P})g}|\\[3pt]
		&+\Big|\inxv{\p_x^\a\p_v^\b\Big\{\big[b\cdot(v-u)\big]\big[\p_t u\cdot(v-u)\M\big]\Big\},\p^\a_x\p_v^\b(\mathbf{I-P})g}\Big|,\\[5pt]
		J_6=&\Big|\inxv{\p_x^\a\p_v^\b\Big\{\big[a(v-u)\cdot \n_x u\big]\cdot(v-u)\M\Big\},\p^\a_x\p_v^\b(\mathbf{I-P})g}\Big|,\\[5pt]
		J_7=&\Big|\inxv{\p_x^\a\p_v^\b\Big\{\big[(v-u)\cdot\n_xb\big]\cdot(v-u)\M\Big\},\p^\a_x\p_v^\b(\mathbf{I-P})g}\Big|,\\[5pt]
		J_8=&\big|\inxv{\p_x^\a\p_v^\b\big[(v\cdot\n_xu)\cdot b\M\big],\p^\a_x\p_v^\b(\mathbf{I-P})g}\big|\\[3pt]
		&+\big|\inxv{\p_x^\a\p_v^\b\big[b\cdot(v-u)(v\cdot \n_xu)\cdot(v-u)\M\big],\p^\a_x\p_v^\b(\mathbf{I-P})g}\big|,\\[5pt]
		J_9=&\Big|\inxv{\p_x^\a\p_v^\b\Big\{\big[(v-u)\cdot \n_x u\big]\cdot(v-u)\M\Big\},\p^\a_x\p_v^\b(\mathbf{I-P})g}\Big|,\\[5pt]
		J_{10}=&\big|\inxv{\p_x^\a \p_v^\b\left[\frac{1}{2}(\mathbf{I-P})g(v-u)\cdot \p_t u\right],\p^\a_x\p_v^\b(\mathbf{I-P})g}\big|,\\[5pt]
		J_{11}=&\big|\inxv{\p_x^\a\p_v^\b\left[\frac{1}{2}(\mathbf{I-P})g(v\cdot \n_xu)\cdot(v-u)\right],\p^\a_x\p_v^\b(\mathbf{I-P})g}\big|,\\[5pt]
		J_{12}=&\big|\inxv{\p_x^\a\p_v^\b\big[\mathrm{div}_xb\M+(b\cdot \n_xu)\cdot(v-u)\M\big],\p^\a_x\p_v^\b(\mathbf{I-P})g}\big|,\\[5pt]
	\end{aligned}
	\end{equation*}
	where the summations $\sum\limits_{1 \leq |\lambda| \leq |\alpha|}$ and $\sum\limits_{1 \leq |\theta| \leq |\beta|}$ are suppressed for conciseness.
		
		In what follows, we only present the detailed estimates for $J_1-J_4$, while the estimates for $ J_{5}-J_{12}$ are analogous to those in Lemma \ref{lem:mi H3x}.

        First, for $J_1$, we integrate by parts. For $|\beta'|=1$ and $\beta'\leq \beta$, we have $|\beta-\beta'|\leq 1$, and
        \begin{equation}\label{eq:J1}
		\begin{aligned}
		J_1=&\frac{1}{\v}|C^{\lambda}_\a\inxv{\p^\b_v[\p^\lambda_x(\frac{|v-u|^2}{4})\p^{\a-\lambda}_x(\mathbf{I-P})g],\p^\a_x\p^\b_v(\mathbf{I-P})g}|\\
        =&\frac{1}{\v}|C^{\lambda}_\a\inxv{\p^{\b-\beta'}_v[\p^\lambda_x(\frac{|v-u|^2}{4})\p^{\a-\lambda}_x(\mathbf{I-P})g],\p^\a_x\p^{\b+\beta'}_v(\mathbf{I-P})g}|\\
		\leq&\delta\frac{\|\p^\a_x\p^\b_v(\mathbf{I-P})g\|_\nu^2}{\v}+C_\delta (T)\sum\limits_{ |\a|\leq 2}\frac{\Lv{\px^\a \mi g}^2}{\varepsilon}.
		\end{aligned}
		\end{equation}
		For $J_2$, noting that $|\beta-\theta|\leq 1 $, we have
		\begin{equation}\label{eq:J2}
		\begin{aligned}
		J_2=&\frac{1}{\v}|C^{\theta}_\b\inxv{\p^\theta_v(\frac{|v-u|^2}{4})\p^{\b-\theta}_v\p^\a_x(\mathbf{I-P})g,\p^\a_x\p^\b_v(\mathbf{I-P})g}|\\[5pt]
		\leq&\delta\frac{\|\p^\a_x\p^\b_v(\mathbf{I-P})g\|_\nu^2}{\v}+C_\delta\frac{\|\p_x^\a(\mathbf{I-P})g\|_\nu^2}{\v}.
		\end{aligned}
		\end{equation}
		Similarly, for $J_3$, we have
		\begin{equation}\label{eq:J3}
		\begin{aligned}
		J_3=&|\inxv{\p^\a_x\p^\b_v(v\cdot \n_x(\mathbf{I-P})g),\p^\a_x\p^\b_v(\mathbf{I-P})g}|\\[5pt]
		\leq&\delta\frac{\|\p^\a_x\p^\b_v(\mathbf{I-P})g\|_\nu^2}{\v}+\v C_\delta (T)\sum\limits_{|\chi|\leq 3}\|\p_x^\chi(\mathbf{I-P})g\|_\nu^2,
		\end{aligned}
		\end{equation}	
        where we use
       	\begin{equation*}
		\begin{aligned}
       |\inxv{v\cdot (\p^\a_x\p^\b_v \n_x(\mathbf{I-P})g),\p^\a_x\p^\b_v(\mathbf{I-P})g}|=0.
       \end{aligned}
		\end{equation*}	
		For $J_4$, we have
		\begin{equation}\label{eq:J4}
		\begin{aligned}
		J_4=&|\inxv{\p^\a_x\p^\b_v \left( \intv{(v-u)v\cdot \n_x((\mathbf{I-P})g\M)}\cdot (v-u)\M \right),\p^\a_x\p^\b_v(\mathbf{I-P})g}|\\[5pt]
		\leq&\delta\frac{\|\p^\a_x\p^\b_v(\mathbf{I-P})g\|_\nu^2}{\v}+\v C_\delta (T)\sum\limits_{|\chi|\leq 3}\|\p_x^\chi(\mathbf{I-P})g\|_\nu^2.\\
		\end{aligned}
		\end{equation}
       For $J_{5}-J_{12} $, we can similarly find
       \begin{equation}\label{eq:J5}
        \begin{aligned}
         \sum\limits_{5\leq i \leq 12} J_{i}\leq \delta\frac{\|\p^\a_x\p^\b_v(\mathbf{I-P})g\|_\nu^2}{\v}+C_\delta (T)\E(t)+C_\delta (T)\v .
        \end{aligned}
       \end{equation}

	Hence, substituting \eqref{eq:J1}--\eqref{eq:J5} into \eqref{eq:+remainder} and summing over $|\alpha|\leq 2$, $1\leq |\beta|\leq 2$, we choose $\delta$ sufficiently small to obtain
		\begin{multline*}
		\sum\limits_{0\leq |\a|\leq 2,1\leq|\b|\leq 2}\left(\dt{t}\|\p^\a_x\p^\b_v(\mathbf{I-P})g\|_{L^2_x L^2_v}^2+\frac{C_1\|\p^\a_x\p^\b_v(\mathbf{I-P})g\|_\nu^2}{\v}\right)\\[5pt]
		\leq C(T)\E(t)+C(T)\v+ \v C (T)\sum\limits_{|\chi|\leq 3}\|\p_x^\chi(\mathbf{I-P})g\|_\nu^2 +C(T)\sum\limits_{ |\a|\leq 2}\frac{\|\p_x^\a(\mathbf{I-P})g\|_\nu^2}{\v},
		\end{multline*}
		   where we have used \eqref{eq:xv mi} in Lemma \ref{lem:property mi}. Integrating the above over $[0,t]$ and combining the result with \eqref{eq:ma g mix}, we can finally obtain \eqref{eq:mi xv}. This completes the proof.
	\end{proof}
	
After obtaining all required estimates,  we are now in a position to present the total energy inequality, which is a direct consequence of combining Corollary \ref{coro:ma g mix}, Lemma \ref{lem:mi xv}, and the definition of $\E(t)$ in \eqref{defE}.	

\begin{proposition}\label{prop:total energy} Under the conditions of Theorem \ref{thm1} and the \textit{a priori} assumption \eqref{assumption}, for $t\in [0,T_1] $, one has
	\begin{equation}\label{eq:total nergy}
	\E(t)\leq C(T)\E(0)+C(T)\varepsilon+C(T)\intt{\E(s)}{s},
	\end{equation}
      where the energy functional $\E(t)$ is given in \eqref{defE} and $C(T) > 1$ is a constant only depending on $T$.
\end{proposition}	

%%%%%%%%%%%%%%%%%%%%%%%%%%%%

\subsection{Proof of main theorem}
\label{subsec:proof thm1}

Applying Gr\"onwall's inequality to \eqref{eq:total nergy}, we obtain
$${\E(t)\leq C(1+T)C(T)e^{C(T)}(\E(0)+\varepsilon)\leq C(T)(\E(0)+\varepsilon)},$$
for all $t\in  [0,T_1]$.
Under the conditions of Theorem \ref{thm1}, there exists $\v_0=\v_0(T)>0$ such that for all $0<\v \leq\v_0$, we have
\begin{equation} \label{eq:E}
\E(t)\leq C(T)(\E(0)+\varepsilon)\leq C(T)(C\v^\gamma+\varepsilon)\leq \frac{3}{2}\eta^2,
\end{equation}
for $t \in [0,T_1]$. Next, we will prove that \eqref{eq:E} holds for all $ t\in \big[0,\min\{T^\v_{max},T\}\big]$. We define
\begin{equation*}
    T^*=\sup\limits_{0\leq t<\min\{T^\v_{max},T\}}\left\{ t\,\big|\, \E(t)\leq 2{\eta^2}\right\}.
\end{equation*}
Noting that $\E(0)\leq C\v ^\gamma$ and recalling \eqref{eq:varepsilon}, i.e., $C\varepsilon^\gamma\leq \eta^2$, we have $T^*\in \big( 0,\min\{T^\v_{max},T\} \big]$. To prove by contradiction, we suppose $T^* <\min\{T^\v_{max},T\}$.
By the definition of $T^*$, we have $\E(t) \leq 2\eta^2$ for all $t \in [0, T^*]$ and particularly $\E(T^*) = 2\eta^2$. By further using \eqref{eq:E}, it actually yields $\E(t)\leq \frac{3}{2}\eta^2$ for all $t \in [0,T^*]$, which contradicts $\E(T^*) = 2\eta^2$.

Therefore, we must have $T^* = \min\{T^\varepsilon_{max},T\}$. Consequently, for all $t \in \big[0, \min\{T^\varepsilon_{max},T\} \big]$, \eqref{eq:E} holds.
Using standard continuity arguments, combined with \eqref{eq:regu limit}, the local existence result in Proposition \ref{prop:local_existence} and the uniform estimates \eqref{eq:E} on $t\in \big[ 0,\min\{T^\varepsilon_{max},T\} \big]$, we conclude that \eqref{eq:E} holds on $[0,T]$.

In view of the definition of $\E(t)$ given in \eqref{defE}, we further deduce that
	\begin{equation*}
	\Hxxx{a-\rho_0}^2+\Hxxx{b}^2+\Hxxx{u-u_0}^2+\Hxxvv{\mi g}^2\leq C(T)(\v^\gamma+\varepsilon).
	\end{equation*}	
		Recalling the decomposition \eqref{MM}, we have
		\begin{equation}\label{decayf}
		\begin{aligned}
		&|f^\v(t,x,v)-(1+\rho_0(t,x))M_{u_0}(t,x,v)|^2+|u^\v(t,x)-u_0(t,x)|^2\\[5pt]
		\le&|M_u-M_{u_0}|^2+|(a-\rho_0)M_{u_0}|^2+|a(M_u-M_{u_0})|^2+|b\cdot(v-u)M_u|^2\\[3pt]
		& \qquad \qquad \qquad \qquad\qquad\qquad\qquad\quad +|(\mathbf{I-P})g\M|^2+|u^\v(t,x)-u_0(t,x)|^2\\[5pt]
		\le&|M_u-M_{u_0}|^2+\Hxx{a-\rho_0}^2+\|b\|_{H^2_x}^2+\|(\mathbf{I-P})g\|^2_{H^2_x H^2_v}+\|u-u_0\|_{H^2_x}^2\\[5pt]
		\le&|M_u-M_{u_0}|^2+C(T)(\v^\gamma+\v).
		\end{aligned}
		\end{equation}
		
		Furthermore, in terms of $|M_u-M_{u_0}|$, we introduce an auxiliary function $G$
		\[
		G(s) := \frac1{(2\pi)^{3/2}}\exp\left\{-\frac{1}{2}\left|v-[u_0(t,x)+s(u(t,x)-u_0(t,x))]\right|^2 \right\},
		\]
		which satisfies $G(1)=M_u$ and $ G(0)=M_{u_0}$.
		By the Newton-Leibniz formula, we have
		\begin{equation}\label{M}
		\begin{aligned}
		\left| M_u - M_{u_0} \right|^2 = \left|G(1) - G(0)\right|^2 =&  \left| \int_{0}^{1}\dt{s}G(s) \,\mathrm{d}s \right|^2\\
		\le & \left(\int_{0}^{1} \left|  G(s)(v-[u_0+s(u-u_0)]) \cdot (u-u_0) \right| \,\mathrm{d}s\right)^2\\
		\le & \left(\int_{0}^{1} |(u-u_0)|\mathrm{d}s\right)^2\le |u-u_0|^2\le C(T)(\v^\gamma+\v),
		\end{aligned}
		\end{equation}
		where we use the fact that $|G(s)(v-[u_0+s(u-u_0)])|\leq C$ in the second inequality above.
		
		Finally, the proof of Theorem \ref{thm1} can be completed by substituting \eqref{M} into \eqref{decayf}.

\appendix

\section{Preliminary lemmas}

In this subsection, we present some preliminary lemmas and their proofs.

\begin{lemma}[\cite{Bobkov2009}, also Lemma 2.2 in \cite{liwangwang}]\label{borel}
	Assume that $\mathcal{B}$ is a Borel probability measure on $\mathbb{R}^n$ and $h$ is a bounded smooth function on $\mathbb{R}^n$. Then for any non-negative, Borel measurable function $\mathcal{W}$, the following weighted Poincar\'e - type inequality holds,
	\begin{equation*}\label{eq:2.18}
	\int |\nabla h|^2 \mathcal{W}^2 \,\d\mathcal{B} \ge \int h^2 \,\d\mathcal{B} - \left( \int h \,\d\mathcal{B} \right)^2.
	\end{equation*}
\end{lemma}

The following lemma is to show the dissipative property of the linearized Fokker--Planck operator $\mathbf{L}$ in $\eqref{sourse}_1$ around the local Maxwellian $M_u$ defined in \eqref{Mlocal}.
\begin{lemma}\label{lemma A2}
	For any $u\in\R^3$, the linearized operator $\mathbf{L}$ defined in $\eqref{sourse}_1$  has the dissipative property, i.e., if $\intxv {g\mathbf{L}g}$ is well-defined, then there exists a constant $c>0$ independent of $u\in\R^3$ and $g$ such that
	\begin{equation*}
	c\|(\mathbf{I-P})g\|_{\nu}^2+\Lx{b}^2 \leq \intxv {g\mathbf{L}g},
	\end{equation*}
    where $b$ and $\mi g$ are defined by \eqref{eq:def-a-b} and \eqref{decom}.
\end{lemma}

\begin{proof}
	By a direct computation, we have
	\begin{equation*}
	\begin{aligned}
	\intv{g \mathbf{L} g}=& -\intv{\frac{g}{\M}\div_v\left( \M \nabla_v\left( \frac{g}{\M}\right)\right)}\\
	=& \intv{\left|\nabla_vg+\frac{1}{2}(v-u)g\right|^2}\\
	=& \intv{\left(\left| \nabla_v g \right|^2+\frac{|v-u|^2}{4}g^2-\frac{3}{2}g^2\right)},
	\end{aligned}
	\end{equation*}
	and replacing $ g$ by $(\mathbf{I-P})g$ yields
	\begin{equation}\label{:eq:diss1}
	\intv{(\mathbf{I-P})g\, \mathbf{L(\mathbf{I-P})g}}=\intv{\left(|\nabla_v(\mathbf{I-P})g|^2+\left(\frac{|v-u|^2}{4}-\frac{3}{2}\right)|(\mathbf{I-P})g|^2\right)}.
	\end{equation}
	By considering the Borel probability measure $\d\mathcal{B}=M_u\, \d v $ and weighted function $\mathcal{W}=1$ in Lemma \ref{borel}, and setting $ h=\dfrac{(\mathbf{I-P})g}{\M}$, we have
	\begin{equation}\label{eq:diss2}
	\begin{aligned}
	\intv{(\mathbf{I-P})g\, \mathbf{L(\mathbf{I-P})g}} =& \intv{\left|\nabla_v(\mathbf{I-P})g+\frac{1}{2}(v-u)(\mathbf{I-P})g\right|^2} \\
	\ge& \intv{|(\mathbf{I-P})g|^2}.
	\end{aligned}
	\end{equation}
	Applying $\eqref{:eq:diss1} + \eqref{eq:diss2} \times 2 $ yields
	\begin{equation}\label{:eq:diss (I-P)}
	\begin{aligned}
	\intv{(\mathbf{I-P})g\, \mathbf{L(\mathbf{I-P})g}}\ge& \frac{1}{12}\intv{\left(|\nabla_v(\mathbf{I-P})g|^2+\left(|v-u|^2+1\right)|(\mathbf{I-P})g|^2\right)}.
	\end{aligned}
	\end{equation}
	By further substituting the decomposition \eqref{decom} of $g$, we find
	\begin{equation*}
	\begin{aligned}
	\intv{g\mathbf{L}g}&=\intv{\left(a\M +b\cdot(v-u)\M+(\mathbf{I-P})g\right)\left(b\cdot(v-u)\M+\mathbf{L}(\mathbf{I-P}g)\right)}\\
	&=|b|^2+\intv{(\mathbf{I-P})g\, \mathbf{L(\mathbf{I-P})g}}.
	\end{aligned}
	\end{equation*}
	Integrating over $x \in \R^3 $ and combining \eqref{:eq:diss (I-P)}, we obtain
	\begin{equation*}
	\frac{1}{12}\|(\mathbf{I-P})g\|_{\nu}^2+\Lx{b}^2 \leq \intxv {g\mathbf{L}g}.
	\end{equation*}
	This completes the proof.
	
\end{proof}

\begin{lemma}\label{lem:property mi}
Under the conditions of Theorem \ref{thm1} and the \textit{a priori} assumption \eqref{assumption}, for $t\in [0,T_1] $, we have
	\begin{equation}\label{eq:xv mi}
	\begin{aligned}
	\|\p^\a_x\p^\b_v(\mathbf{I-P})g\|_\nu^2\leq C\|(\mathbf{I-P})\p^\a_x\p^\b_v(\mathbf{I-P})g\|_\nu^2+C(T)\|(\mathbf{I-P})g\|_{H^2_x L^2_v}^2,
	\end{aligned}
	\end{equation}
	where $0\leq |\a|\leq 2$ and $1\leq |\b|\leq 2$, and
	\begin{equation}\label{eq:x mi}
	\begin{aligned}
	\|\n^k_x(\mathbf{I-P})g\|_\nu^2\leq C\|(\mathbf{I-P})\n^k_x(\mathbf{I-P})g\|_\nu^2+C(T)\sum\limits_{|\a|\leq k-1}\|\p_x^\a(\mathbf{I-P})g\|_\nu^2,
	\end{aligned}
	\end{equation}
	for $k=1,2,3$.
\end{lemma}

\begin{proof}
	For the mixed derivatives $\p^\a_x\p^\b_v(\mathbf{I-P})g$ with $0\leq |\a|\leq 2$ and $1\leq |\b|\leq 2$, expanding the projection yields
	\begin{equation*}
	\begin{aligned}
	(\mathbf{I-P})\p^\a_{x}\p^\b_v(\mathbf{I-P})g=&\p^\a_{x}\p^\b_v(\mathbf{I-P})g+(-1)^{|\b|}\intv{\p^\b_v(\M)\p^\a_x(\mathbf{I-P})g}\M\\
	&+(-1)^{|\b|}\intv{\p^\b_v\big[(v-u)\M\big]\p^\a_x(\mathbf{I-P})g}\cdot (v-u)\M,
	\end{aligned}
	\end{equation*}
	which implies
	\begin{equation*}
	\begin{aligned}
	\|\p^\a_x\p^\b_v(\mathbf{I-P})g\|_\nu^2\leq C\|(\mathbf{I-P})\p^\a_x\p^\b_v(\mathbf{I-P})g\|_\nu^2+C\|(\mathbf{I-P})g\|_{H^2_x L^2_v}^2.
	\end{aligned}
	\end{equation*}
	
	Regarding the first-order spatial derivatives $\n_x(\mathbf{I-P})g$, we have
	\begin{equation*}
	\begin{aligned}
	&(\mathbf{I-P})\p_{x_i}(\mathbf{I-P})g\\[5pt]
	=&\p_{x_i}(\mathbf{I-P})g + \frac{1}{2} \left( \intv{(v-u)(\mathbf{I-P})g(v-u) \cdot u_{x_i}\M}  \right)  \cdot (v-u)\M,
	\end{aligned}
	\end{equation*}
	which leads to the following estimate:
	\begin{equation*}
	\begin{aligned}
	\|\n_x(\mathbf{I-P})g\|_\nu^2\leq C\|(\mathbf{I-P})\n_x(\mathbf{I-P})g\|_\nu^2+C(T)\|(\mathbf{I-P})g\|_\nu^2.
	\end{aligned}
	\end{equation*}
	
	For the second-order spatial derivatives $\n^2_x(\mathbf{I-P})g$, an analogous expansion gives
	\begin{equation*}
	\begin{aligned}
	&(\mathbf{I-P})\p_{x_ix_j}(\mathbf{I-P})g\\[5pt]
	=&\p_{x_ix_j}(\mathbf{I-P})g+\frac{1}{2}\intv{\p_{x_i}(\mathbf{I-P})g(v-u)\cdot \p_{x_j}u\M}\M\\
	&+\frac{1}{2}\intv{\p_{x_j}(\mathbf{I-P})g(v-u)\cdot \p_{x_i}u\M}\M+\frac{1}{2}\intv{(\mathbf{I-P})g\p_{x_ix_j}\M}\M\\
	&+\frac{1}{2}\intv{\p_{x_i}(\mathbf{I-P})g\p_{x_j}\big[(v-u)\M\big]}\cdot(v-u)\M\\
	&+\frac{1}{2}\intv{\p_{x_j}(\mathbf{I-P})g\p_{x_i}\big[(v-u)\M\big]}\cdot(v-u)\M\\
	&+\frac{1}{2}\intv{(\mathbf{I-P})g\p_{x_ix_j}\big[(v-u)\M\big]}\cdot(v-u)\M,
	\end{aligned}
	\end{equation*}
	which yields the bound:
	\begin{equation*}
	\begin{aligned}
	\|\n^2_x(\mathbf{I-P})g\|_\nu^2\leq C\|(\mathbf{I-P})\n^2_x(\mathbf{I-P})g\|_\nu^2+C(T)\sum\limits_{|\a|\leq 1}\|\p_x^\a(\mathbf{I-P})g\|_\nu^2.
	\end{aligned}
	\end{equation*}
	
	By a similar argument, for the third-order derivatives, we obtain
	\begin{equation*}
	\begin{aligned}
	\|\n^3_x(\mathbf{I-P})g\|_\nu^2\leq C\|(\mathbf{I-P})\n^3_x(\mathbf{I-P})g\|_\nu^2+C(T)\sum\limits_{|\a|\leq 2}\|\p_x^\a(\mathbf{I-P})g\|_\nu^2.
	\end{aligned}
	\end{equation*}
	This completes the proof.
\end{proof}

\begin{lemma}\label{lem:hx}
	Under the conditions of Theorem \ref{thm1} and the \textit{a priori} assumption \eqref{assumption}, for $t\in [0,T_1] $, we have, for $1\leq |\a|\leq 3$,
	\begin{equation*}
	\begin{aligned}
	&\varepsilon\Lxv{\px^\a h}^2\leq C(T)\varepsilon+C\E(t).
	\end{aligned}
	\end{equation*}
\end{lemma}	

\begin{proof}
	For $1\leq |\a|\leq 3$, substituting the macroscopic-microscopic decomposition of $h$ yields
	\begin{equation*}
	\begin{aligned}
	&\varepsilon\Lxv{\px^\a h}^2\\
	=&\varepsilon\Lxv{\px^\a(\rho\M+b\cdot(v-u)\M+\mi h)}^2\\[5pt]
	\leq&\varepsilon\Lxv{\px^\a((\phi+\rho_0+1)\M)}^2+\varepsilon\Lxv{\px^\a(b\cdot(v-u)\M)}^2+\varepsilon\Lxv{\px^\a(\mi h)}^2\\[5pt]
	\leq&C\varepsilon(1+\Hxxx{\psi}+\Hxxx{\rho_0})(1+\Hxxx{\psi}^3+\Hxxx{u_0}^2)+C\varepsilon\Lxv{\px^\a(\mi h)}^2\\[5pt]
	&+C\varepsilon(\Hxxx{b}+1)(\Hxxx{\psi}^3+\Hxxx{u_0}^3+1)\\[5pt]
	\leq&C(T)\varepsilon+C\E(t).
	\end{aligned}
	\end{equation*}
		This completes the proof.
\end{proof}

\begin{lemma}\label{lem:b-b_a}		Under the conditions of Theorem \ref{thm1} and the \textit{a priori} assumption \eqref{assumption}, for $t\in [0,T_1] $, we have, for $1\leq |\a|\leq 4$:
	\begin{equation*}
	\varepsilon\Lx{\px^\a b-b_\a}^2\leq C(T)\varepsilon+C(T)\E(t).
	\end{equation*}
\end{lemma}	

\begin{proof}
	Recalling the definitions
	\begin{equation*}
	b_\a=\int \px^\a h (v-u) \M \d v,\ b=\int h(v-u)\M \d v,
	\end{equation*}
	we apply the Leibniz rule and Lemma \ref{lem:hx} to deduce
	\begin{equation*}
	\begin{aligned}
	&\varepsilon\Lx{\px^\a b-b_\a}^2\\
	\leq& C \sum\limits_{ 1\leq|\b|\leq|\a|-1}\varepsilon\left\|\int \px^\b ((v-u)\M)\px^{\a-\b}h \d v \right\|^2_{L^2_x}+\varepsilon\left\|\int \px^\a ((v-u)\M)\,h \d v \right\|^2_{L^2_x}\\[5pt]
	\leq&\varepsilon \big( \Hxxx{u_0}^4+\Hxxx{\psi}^4+1 \big) \big(\Hxxv{\nabla_x h}^2+\|h\|^2_{L^\infty_xL^2_v}\big) + C\varepsilon\Hxxx{\nabla_xu}^2\|h\|^2_{L^\infty_xL^2_v}\\[5pt]
	\leq&C(T)\varepsilon+C(T)\E(t).
	\end{aligned}
	\end{equation*}
	This completes the proof.
\end{proof}

\section{Well-posedness and regularity of limiting system}

In this subsection, we present some results of well-posedness and regularity for the limiting system, which are repeatedly used in the proof above.

\begin{proposition}(\cite[Theorem 1.2]{Paicu2013})\label{prop:limit}
	Given the initial data $(\rho^{in}_0, u^{in}_0)$ satisfying
	\begin{equation*}
	0 \leq \rho^{in}_0(x) \leq \overline{\rho} < +\infty,\quad u^{in}_0 \in H^1_x(\mathbb{R}^3),
	\end{equation*}
	there exists a constant $\eta_0 > 0$ depending only on $\overline{\rho}$ such that if
	\begin{equation*}
	\|u^{in}_0\|_{L^2_x}\|\n_x u^{in}_0\|_{L^2_x} \leq \eta_0,
	\end{equation*}
	the system \eqref{limit} admits a unique global solution $(\rho_0, u_0)$ satisfying, for $(t, x) \in [0, +\infty) \times \mathbb{R}^3$
    \[
    0 \leq \rho_0(t, x) \leq \overline{\rho},
    \]
    and for $ t \in [0, +\infty)$,
	\begin{equation*}
	\begin{aligned}
	&B_0(t) \leq \|\sqrt{\rho^{in}_0+2}u^{in}_0\|_{L^2_x}, \qquad B_1(t) \leq 2\|\n_xu_0\|_{L^2_x},\\[5pt]
	&	B_2(t) \leq C\bigl(1 + \|\nabla u^{in}_0\|_{L^2_x}^4\bigr) \|\nabla u^{in}_0\|_{L^2_x}^2 \exp\bigl\{C\bigl(\|u_0^{in}\|_{L^2_x}^2 + \|\nabla u^{in}_0\|_{L^2_x}^2\bigr)\bigr\},
	\end{aligned}
	\end{equation*}
	where $ C$ is a constant depending on $\overline{\rho}$, and $B_0(t),\ B_1(t),\ B_2(t)$ are defined
	by
	\begin{align*}
	B_0(T) &:= \sup_{t \in [0,T]} \int_{\mathbb{R}^3} (\rho_0+2)|u_0(t, x)|^2 dx + 2 \int_0^T \int_{\mathbb{R}^3} |\n_x u_0|^2 \d x \d t,\\
	B_1(T) &:= \sup_{t \in [0,T]} \int_{\mathbb{R}^3} |\n _xu_0(t, x)|^2 dx + 2 \int_0^T \int_{\mathbb{R}^3} \big[(\rho_0+2)|\p_t u_0|^2 + |\n^2_x u_0|^2\big] \d x \d t,\\
   B_2(T)&:= \sup_{t\in[0,T]} \int \min\{1,t\}\left( (\rho_0+2)|\pt u_{0}(t,x)|^2 + |\nabla^2_x u_0(t,x)|^2 +|\nabla_x P_0|^2 \right)\d x \\
   &\quad+ \int_{0}^{T} \int \min\{1,t\}|\nabla_x \pt u_{0}|^2 \d x \d t.
	\end{align*}
\end{proposition}

\begin{corollary}\label{coro:limit}
	Under the conditions of Proposition \ref{prop:limit}, assume that the initial condition $(\rho^{in}_0, u^{in}_0)$  of \eqref{limit} satisfies
	\begin{equation*}
	(\rho^{in}_0, u^{in}_0)\in H^6_x(\mathbb{R}^3),
	\end{equation*}
	then, the solution $(\rho_0,u_0)$ to the system $\eqref{limit}$ satisfies
\begin{equation*}
\partial^k_t \rho_0, \quad \partial^k_t u_0 \in  C\big([0,T];H^{6-2k}_x(\mathbb{R}^3)\big)
\end{equation*}
for \(k=0,1,2\) and all \(T>0\).
\end{corollary}	
\begin{proof}
	 Given the global strong solution established in Proposition \ref{prop:limit}, the higher-regularity estimates follow from a standard bootstrap argument. Specifically, these bounds are obtained by alternately applying the method of characteristics to the transport equation and the classical parabolic regularity theory to the momentum equation (see \cite{CK2003}). Here, $T$ can be any arbitrary time.
\end{proof}
%===========================================================
	\subsection*{Statements and Declarations.}
	The authors declare that they have no conflict of interest. Data sharing is not applicable to this article. Generative AI was used only for English-language editing. All scientific content, analysis, results, and conclusions were produced and verified by the authors.
	
	\section*{Acknowledgments}
	The work of H.~Wen is supported by the National Natural Science Foundation of China 12471209, and by the Guangzhou Basic and Applied Research Projects SL2024A04J01206. K.~Qi would like to thank the support from AMS-Simons travel award and startup fund from Michigan State University.
	
	%-------------------------------

\end{document}